\documentclass[a4paper,twoside]{amsart}

\usepackage{ dsfont }
\usepackage{lmodern}
\usepackage[T1]{fontenc}
\usepackage{babel}
\usepackage[utf8]{inputenc}
\usepackage{amsmath}
\usepackage{amsthm}
\usepackage{amsfonts}
\usepackage{graphicx,caption}
\usepackage{color} 
\usepackage[all]{xy}
\usepackage{multicol}
\usepackage{amssymb}
\usepackage[percent]{overpic}
\usepackage{setspace}
\usepackage{fancybox, calc}
\usepackage[some]{background}

\newcommand{\obra}[3]{{\sc #1} {\em #2}. {#3}.}

\theoremstyle{definition}
\newtheorem{defin}{Definition}[section]
\newtheorem{ejem}[defin]{Example}
\newtheorem{teo}[defin]{Theorem}

\newtheorem{prop}[defin]{Proposition}
\newtheorem{lema}[defin]{Lemma}
                   
\theoremstyle{remark}
\newtheorem{rem}[defin]{Remark}
\newtheorem{notation}[defin]{Notation}

\graphicspath{{imagenes/}}

\begin{document} 

\title[Stratified Reduction of Singularities]{Stratified Reduction of Singularities of Generalized Analytic Functions}
\author{B. Molina-Samper}
\address{Departamento de Matemáticas. Universidad Autónoma de Madrid,  Spain.}
\email{beatriz.molina@uam.es}
\author{J. Palma-Márquez}
\address{Weizmann institute of science, Rehovot, Israel.}
\email{jesus.palma@weizmann.ac.il}
\author{F. Sanz-Sánchez }
\address{Dpto. \'Algebra, An\'alisis Matem\'atico, Geometr\'ia y Topolog\'ia. Universidad de Valladolid, Spain}
\email{fsanz@agt.uva.es}
\thanks{First and third authors are partially supported by the Project ``M\'etodos asint\'oticos, algebraicos y geom\'etricos en foliaciones singulares y sistemas din\'amicos'' (Ref.:  PID2019-105621GB-100) of the Ministerio de Ciencia in Spain. First author is partially supported by the ``Programa de becas postdoctorales DGAPA'' of the UNAM in México . Second author is  partially supported by Israel Science Foundation grant 1167/17; by funding received from the MINERVA Stiftung with funds from Germany’s Federal Ministry of Education and Research (BMBF); and Papiit Dgapa UNAM IN110520}

\date{}
\dedicatory{}

\maketitle

\begin{abstract}

Generalized analytic functions over generalized analytic manifolds are build from sums of convergent real power series with non-negative real exponents (and some well-ordering condition on the support). In a paper by Martín-Villaverde, Rolin and Sanz-Sánchez it is established a result of local reduction of singularities for such a functions. In this paper we deal with a first approach of the global problem. Namely, we prove that a germ of generalized analytic function can be transformed by a finite sequence of blowing-ups with closed centers into a function which is locally of monomial type with respect to the coordinates defining the boundary of the manifold (a normal crossings divisor).
	
\end{abstract}

\section{Introduction}\label{sec:intro}
In this paper, a generalized power series (in $n$ variables and with coefficients in some ring $A$) is a power series with tuples of non-negative real numbers as exponents and whose support is contained in a cartesian product of well-ordered subsets of $\mathbb{R}_+=\{r \geq 0\}$. It is worth to mention that this condition on the support is more restrictive (except for $n=1$) than the one used to define the Hahn ring $A((\Gamma))$, where $\Gamma$ is the group $\mathbb{R}^n$ with the lexicographic order, whose elements are also called generalized power series. Introduced and studied by van den Dries and Speissegger in \cite{vdD-Spe}, generalized power series appear in several contexts. To mention a few: as solutions of differential/functional equations; expression of the Riemann zeta-function (or, more generally, the Dirichlet series) in a logarithmic chart; as asymptotic expansions of Dulac transition maps of vector fields (see for instance \cite{Ily,Kai-Rol-Spe}); in model theory and o-minimal geometry (the paper \cite{vdD-Spe} itself or \cite{Rol-Ser}); as parametrizations of algebraic curves in positive characteristic.
	
Taking real coefficients, we have a natural notion of convergence for generalized power series, whose sums provide continuous functions on open subsets of the orthant $\mathbb{R}_+^n$, called {\em generalized analytic functions}. They are the local pieces to build abstract {\em (real) generalized analytic manifolds}, introduced and developed by Martín, Rolin and Sanz in \cite{Mar-Rol-San}. More precisely, a generalized analytic manifold is a local ringed space $\mathcal{M}=(M,\mathcal{G}_M)$, where $M$ is a topological manifold with boundary and $\mathcal{G}_M$ is a sheaf of continuous functions locally isomorphic to the sheaf of generalized analytic functions on open subsets of $\mathbb{R}_+^n$. Equivalently, as usual, $\mathcal{M}$ can be determined by an atlas of local charts that realize those local sheaves isomorphisms. 
Elements of the sheaf $\mathcal{G}_M$ are called themselves generalized analytic functions on $M$. 

\strut

The main result in \cite{Mar-Rol-San} establishes the local reduction of singularities of generalized analytic functions, in the spirit of Zariski's local uniformization theorem of algebraic varieties \cite{Zar} or Hironaka's version for analytic varieties \cite{Hir1}. The statement, formulated in analogous terms to those used in Bierstone-Milman's paper \cite{Bie-Mil} for real (standard) analytic functions, is the following:
	
\vspace{.2cm}
	
{\bf Local Monomialization Theorem \cite{Mar-Rol-San}.} Let $f$ be a generalized analytic function on $M$ and let $p\in M$. Then there exists a neighborhood $U_0$ of $p$ in $M$, finitely many sequences of {\em local blowing-ups} $\{\pi_i:\mathcal{M}_i\to U_0\}_{i=1}^r$ and compact sets $L_i\subset M_i$ satisfying that $\cup_i\pi_i(L_i)$ is a neighborhood of $p$ and such that, for every $i$, the total transform $f_i=f\circ\pi_i$ is of {\em monomial type} at every $q\in L_i$ (i.e., for some coordinates ${\bf x}=(x_1,x_2,...,x_n)$ centered at $q$, we have $f_i={\bf x}^\alpha U({\bf x})$ where $U(0)\ne 0$).
	
\vspace{.2cm}
The centers of blowing-ups in each sequence $\pi_i$ have normal crossings with the boundary but they are defined only in some open sets of the corresponding manifold.  
In the standard real analytic case, we have stronger global monomialization results (typically called {\em Reduction of Singularities}, see \cite{Aro-Hir-Vic,Bie-Mil2,Enc-Vil}) consisting, essentially, in that in the above statement, we can take just a single sequence ($r=1$) and that the centers of blowing-ups are globally defined closed analytic submanifolds, having normal crossings with the total divisor. 
	
Such a global result is not known so far for generalized analytic functions. There are two main difficulties, related to the very notion of a blowing-up morphism. In one hand, a blowing-up may depend on the local coordinates that we use to write it with the usual expressions. More intrinsically, a blowing-up is not univocally defined and dependens on the use of a {\em standardization} of the manifold (or at least of an open neighborhood of the center of blowing-up). Roughly, it is a subsheaf $\mathcal{O}_M$ of $\mathcal{G}_M$ such that $(M,\mathcal{O}_M)$ is a real analytic standard manifold and from which the sheaf $\mathcal{G}_M$ can be recovered by a natural completion adding generalized series (see \cite{Mar-Rol-San}, we recall this notion below). 	Secondly, although every generalized analytic manifold is locally standardizable, there may exists closed submanifolds, having normal crossings with the boundary, which do not admit standardizable neighborhoods; i.e., such submanifolds cannot be  ``geometric'' centers for a blowing-up (cf. \cite[Example 3.20]{Mar-Rol-San}).
	
Morally, a procedure for reduction of singularities of generalized analytic functions would need to guarantee  that, in the process, all closed centers susceptible to be blown-up that appear (for instance a submanifold contained in the set of points where the function $f$ has not normal crossing), have standardizable neighborhoods. If this is already proved and $Y$ is such a center, one needs to show furthermore that, among the different standardizations around $Y$, there exists one of them for which the corresponding blowing-up $\pi:\tilde{\mathcal{M}}\to\mathcal{M}$ fits appropriately for a given local strategy (for instance, that certain tuple of numerical invariants decreases lexicographically when calculated at any point in $\pi^{-1}(Y)$).

\strut

In this paper, we overcome these difficulties to obtain an intermediary step towards the global result, the so called {\em stratified} reduction of singularities. Let us explain it. First, we recall that, by its very definition, the boundary $\partial M$ of a generalized analytic manifold is a normal crossing divisor; i.e., $\partial M$ is locally defined at each point $p$ by an equation $x_1x_2\cdots x_e=0$ where the $x_i$ are (generalized analytic) coordinates of a local chart at $p$. Moreover, the number $e=e(p)$ of {\em components} passing through $p$ does not depend on the chart. Extending the map $e:M\to\{0,1,2,\cdots,n\}$ to take the value $e=0$ on the interior $M\setminus\partial M$ of $M$, and taking the partition of $M$ formed by the connected components of the fibers of $e$, we obtain a natural stratification of $M$ by (standard) analytic manifolds. A generalized analytic function $f:M\to\mathbb{R}$ is said to be of {\em stratified monomial type} if for any given $p\in M$, if $S$ is the stratum where $p$ belongs, there exists a local chart  $(\mathbf{x}=(x_1,x_2,...,x_e),\mathbf{y})$ centered at $p$ satisfying $S=\{x_1=x_2=\cdots =x_e=0\}$ and for which 
$$
f(\mathbf{x},\mathbf{y})=\mathbf{x}^\alpha U(\mathbf{x},\mathbf{y}),\;\mbox{ where }\alpha\in\mathbb{R}_+^e\mbox{ and }U(0,\mathbf{y})\not\equiv 0.
$$
Thus, requiring a function to be of stratified monomial type means to require that it is of monomial type only with respect to the generalized coordinates determining equations of the components of the boundary. In particular, the condition is empty if $p\not\in\partial M$. Also, it is automatic if $S$ has codimension $e=1$, taking in the above definition $\alpha$ to be the minimum of the support of the series defining $f$ with respect to the single variable $\mathbf{x}=x_1$. 
	
\strut
	
Our main result may be stated now as follows.

\begin{teo}\label{th:main}[Stratified Reduction of Singularities]
	Let $\mathcal{M}=(M,\mathcal{G}_M)$ be a generalized analytic manifold and let $f:M\to\mathbb{R}$ be a generalized analytic function. Let $p\in M$ and assume that the germ of $f$ at $p$ is not zero. Then, there exists a neihborhood $V_p$ of $p$ in $M$ and there exists a sequence of blowing-ups
	$$
	(M_r,\mathcal{G}_{M_r})\stackrel{\pi_{r-1}}{\rightarrow}
	(M_{r-1},\mathcal{G}_{M_{r-1}})\stackrel{\pi_{r-2}}{\rightarrow}
	\cdots
	\stackrel{\pi_2}{\rightarrow}
	(M_1,\mathcal{G}_{M_1})\stackrel{\pi_0}{\rightarrow}
	(V_p,\mathcal{G}_M|_{V_p})
	$$
  	such that the pull-back $\overline{f}:=f\circ\pi_0\circ\cdots\circ\pi_{r-1}\in\mathcal{G}_{M_r}(M_r)$ is of stratified monomial type. Moreover, the (support of the) center of each blowing-up $\pi_j$, $j=0,1,\ldots,r-1$ can be chosen to be the closure of a codimension two stratum in $M_{j}$, where $M_0:=V_p$.	
\end{teo}

The proof of Theorem~\ref{th:main} in this paper is constructive in the sense that each center, as well as the standardization used to blow-up it at each step, can be given explicitly in terms of the expression of $f$ in some initial coordinates of $\mathcal{M}$ at $p$. Moreover, each blowing-up morphism is locally expressed as a purely monomial map between to domains of $\mathbb{R}^n_+$ in convenient charts. Consequently, all the process of stratified reduction of singularities can be described using only combinatorics from the starting data given simply by the minimal support of a generalized power series representing $f$ at $p$. The datum of minimal support is analogous to that of the Newton polyhedron of a function in the standard analytic case and therefore, our result should be compared with the combinatorial reduction of singularities stated in Molina's paper \cite{Mol}. Although it has been a source of inspiration for us, we cannot apply directly the results in that paper, mostly because there is no good notion of ``multiplicity'' in the generalized non-standard situation (any power function with positive real exponent in a generalized variable is a genuine change of variables).  
	
We want to observe that Theorem~\ref{th:main} is already proved for $\dim M=3$ in Palma's paper \cite{Pal}, but with a different strategy for the choice of the sequence of blowing-ups (for instance, the centers of blowing-ups may be either corner points or closures of one-dimensional strata).
	
	\strut
	
The paper is structured as follows. 
	
In section 2 we summarize the basic notions and properties of generalized power series and of the category of generalized analytic manifolds from the mentioned references \cite{vdD-Spe} and \cite{Mar-Rol-San}. We emphasize the notion of standardization, which is crucial to define blowing-ups, for which our approach is slightly different (but equivalent) to the original one from \cite{Mar-Rol-San}.  
	
In section 3 we introduce the category of {\em monomial (generalized analytic) manifolds}, which has many combinatorial properties. The objects of this 	subcategory are those  generalized analytic manifolds, having at least one corner, and equipped with an atlas of local charts at every corner point for which the change of coordinates is expressed as a monomial map between domains of the local model $\mathbb{R}^n_+$. We represent these changes of coordinates by means of a family of matrices of exponents, a combinatorial data which codifies univocally the structual sheaf of the manifold. We define also the subclass of monomial morphisms between monomial manifolds (those that can be expressed as monomial maps in the distinguished atlases) and the class of monomial standardizations of monomial manifolds. These last notion is important because, after a blowing-up using such a standardization with a center which is the closure of a stratum (a so-called {\em combinatorial center}), the new space obtained is again a monomial manifold and the blowing-up morphism is a monomial one. The main result in this section is the abundance of monomial standardizations in a monomial manifold (Proposition \ref{prop:extension}). Furthermore, we have always a global monomial standardization whose expression at a given fixed corner point in the distinguished atlas is prescribed a priori. This permits us, not only to be able to perform a blowing-up with a desired center, but, moreover, to have a prescribed local expression at a given corner point, so that, local strategies of reduction of singularities are susceptible to be ``globalized''. We end this section by introducing a special class of monomial  manifolds, those obtained from a given monomial manifold by a sequence of blowing-ups with combinatorial centers using only monomial standardizations. Such a sequence is called a {\em monomial star} and the family of such stars is called the {\em monomial ``voûte étoilée''}, a terminology that evokes the one introduced by Hironaka in \cite{Hir1, Hir2} for sequences of local blowing-ups in complex analytic geometry.
	
In section 4 we provide a proof of the main Theorem~\ref{th:main}. First, we prove a result of {\em principalization} of monomial finitely generated ideal sheaves in a given monomial manifold. This result (see Theorem \ref{teo:principalizacion})) can be seen as a version for our category of monomial generalized analytic manifolds of a well known result on principalization of ideals in the algebraic or standard analytic situation (see for instance Goward's paper \cite{Gow} for a simple proof, or see also Fernández-Duque's paper \cite{Fer} for a similar statement concerning the resonances elimination for singularities of codimension-one analytic foliations). It should be mention that, taking into account that it suffices to obtain the principalization only at the corner points (cf. above that a function is of stratified monomial type iff it is so at corner points), such a result can also be regarded as a globalization of the algorithm described in vdDries and Speissegger's paper (see \cite[Lemma 4.10]{vdD-Spe}) that reduces the number of elements in the minimal support of a generalized power series by monomial transformations (which represent the local expressions of a local blowing-up with a codimension two center). Although we use certain elements and arguments of that result, and despite of what we have said above concerning the possibility to globalize a ``local strategy'', our proof here uses a different control invariant to achieve: we need to prove that this invariant drops at {\em every} new point after a blowing-up, not only at those in a fiber of the blowing-up morphism over a fixed corner point of the blown-up center.  
	
Once we have the principalization of monomial ideals, the main theorem is concluded easily in the case we start with a point $p$ which is a corner point of the original manifold $\mathcal{M}$. In this case, the sequence $\pi_0\circ\pi_1\circ\cdots\pi_{r-1}$ for Theorem~\ref{th:main} is actually a star in the voûte étoilée over the germ of $M$ at $p$. Finally, the general case $p\in\partial M$ is reduced to the case of a corner point, using the fact that, in general, there is a neighborhood of $p$ with a product structure of a neighborhood of a corner point times an standard analytic manifold without boundary.

\section{Preliminaries}\label{sec:preliminaries}
We summarize here the basic notions about the category of generalized analytic manifolds and blowing-up morphisms in it, introduced by Martín, Rolin and Sanz in \cite{Mar-Rol-San}.  These manifolds are built from convergent generalized power series, extensively studied in a paper by van den Dries and Speissegger \cite{vdD-Spe}.

\subsection{Formal and Convergent Generalized Power Series}
Denote by $\mathbb{R}_+=[0,\infty)$. Tuples of variables are denoted by $X,Y,Z$, etc., and we implicitly assume that tuples with different name have no common variables. If $X$ has $n$ components, we say that $X$ is an $n$-tuple and so on.

Let $X=(X_1,X_2,\ldots,X_n)$ be an $n$-tuple of variables and let $A$ be an integral domain. A {\em formal generalized power series with coefficients in $A$} in the variables $X$ is a map $s:\mathbb{R}_+^n\to A$, written as 
$$
s=\sum\limits_{\lambda\in\mathbb{R}_{+}^n}s_\lambda X^\lambda,\quad X^\lambda=X_1^{\lambda_1}X_2^{\lambda_2}\cdots X_n^{\lambda_n},\;\lambda=(\lambda_1,\lambda_2,\ldots,\lambda_n), \quad s_\lambda=s(\lambda) \in A,
$$ 
such that its \emph{support} $\rm{Supp}(s):=\{\lambda \in \mathbb{R}^n_+: \;s_\lambda\ne 0\}$ is contained in a cartesian product of $n$ well-ordered subsets of $\mathbb{R}$. The set of all such formal generalized power series, denoted by $A[[X^*]]$, with the usual addition and product operations of power series has an structure of $A$-algebra which is also an integral domain. Moreover, if $A$ is a field, then $A[[X^*]]$ is a local algebra (see \cite[Corollary 5.6]{vdD-Spe}), with maximal ideal given by $\mathfrak{m}=\{s \in  A[[X^*]]: \;s_0=0\}$. Note that $A[[X^*]]$ is not noetherian, in fact, the ideal $\mathfrak{m}$ is not finitely generated.

The {\em minimal support} of a power series $s\in A[[X^*]]$ is the subset $\rm{Supp}_{min}(s)\subset \rm{Supp}(s)$ composed of the minimal tuples of $\mathbb{R}^n_{+}$ with respect to the (partial) {\em division order $\leq_d$}, that is
$(\lambda_1,\lambda_2,\ldots,\lambda_n)\leq_d(\mu_1,\mu_2,\ldots,\mu_n)$ if and only if $\lambda_i\leq\mu_i$, for all $i\in \{1,2,\ldots,n\}$.

The condition imposed to the support of a power series $s$ allows to show that the minimal support $\rm{Supp}_{min}(s)$ is finite (see \cite[Lemma 4.2]{vdD-Spe}). As a consequence, $s$ admits a finite {\em monomial presentation}:
$$
s=\sum\limits_{\lambda\in\rm{Supp}_{\min}(s)}X^\lambda\,U_\lambda(X),
$$
where $U_\lambda\in A[[X^*]]$ satisfies $U_\lambda(\mathbf{0})\ne 0$, for any $\lambda \in\rm{Supp}_{\min}(s)$. Denote by $m(s)=\#\rm{Supp}_{\min}(s)$. When $m(s)=1$ or, equivalently, the monomial representation of $s$ has a single term, we say that $s$ is of {\em monomial type}.

\strut

In this paper, we are interested in real generalized power series, that is $A=\mathbb{R}$, but we use different rings when we want to distinguish some variables and put the others into the coefficients. To be precise, if $Y$ and $Z$ are tuples of $k$ and $n-k$ variables, respectively, we consider $\mathbb{R}[[(Y,Z)^*]]$ as a proper $\mathbb{R}$-subalgebra of $\mathbb{R}[[Y^*]][[Z^*]]$ by the natural monomorphism 
\begin{equation}\label{eq:monomorfismo}
s=\sum\limits_{(\lambda,\mu)\in\mathbb{R}_{+}^{n}}a_{\lambda\mu}Y^\lambda Z^\mu\;\mapsto\;
s^Z=\sum\limits_{\mu\in\mathbb{R}_{+}^{n-k}}A_\mu Z^\mu, \quad A_\mu=\sum_{\lambda\in \mathbb{R}^k_+} a_{\lambda\mu}Y^\lambda. 
\end{equation}
If $\text{pr}:\mathbb{R}^{n}\to\mathbb{R}^{n-k}$ denotes the natural projection onto the last $n-k$ coordinates, for any power series $s\in\mathbb{R}[[(Y,Z)^*]]$ we have the inclusion $\text{Supp}_{\min}(s^Z)\subset\text{pr}(\text{Supp}_{\min}(s))$, and as a consequence we get the inequality
\begin{equation} \label{eq:soporteminimoproyeccion}
m(s^Z)\leq m(s).
\end{equation}

Let us write $\mathbb{R}[[Y,Z^*]]$ to denote the subalgebra of the so-called {\em real mixed power series}: those 
formal real generalized power series $s$ in the variables $(Y,Z)$, such that the inclusion $\text{Supp}(s)\subset\mathbb{N}^k\times\mathbb{R}_+^{n-k}$ holds, or equivalently, such that $s^Z\in\mathbb{R}[[Y]][[Z^*]]$. 

\strut

Given an $n$-tuple of variables $X$ and a polyradius $\rho=(\rho_1,...,\rho_n)\in\mathbb{R}_{>0}^n$, denote by $\mathbb{R}\{X^*\}_\rho$ the subalgebra of $\mathbb{R}[[X^*]]$ consisting in those power series $s$ for which
$$
\|s\|_\rho: =\sum_{\lambda\in\text{Supp}(s)}|s_\lambda|\rho^\lambda<\infty.
$$
The union of the $\mathbb{R}\{X^*\}_\rho$ along all the possible polyradius $\rho\in\mathbb{R}_{>0}^n$ is again a subalgebra $\mathbb{R}\{X^*\} \subset \mathbb{R}[[X^*]]$, and its elements are called \emph{(real) convergent generalized power series}. We have that $\mathbb{R}\{X^*\}$ is also a local algebra, whose maximal ideal is given by $\mathfrak{m}\cap \mathbb{R}\{X^*\}$. If $Y, Z$ are tuples of $k$ and $n-k$ variables, respectively, and $\rho\in\mathbb{R}^n_{>0}$ is a polyradius, an element $s\in\mathbb{R}[[Y,Z^*]]\cap\mathbb{R}\{(Y,Z)^*\}_\rho$ gives rise to a continuous function 
\begin{equation}\label{eq:serieconvergente}
\begin{array}{cccc}
	f_s:&P_{k,n-k}^\rho
	&\to&\mathbb{R} \\
	&x=(x_1,x_2,\ldots,x_n)&\mapsto&\sum_{\lambda}s_\lambda x^\lambda,
\end{array}
\end{equation}
where $P_{k,n-k}^\rho=(-\rho_1,\rho_1)\times(-\rho_2,\rho_2)\times\cdots\times(-\rho_k,\rho_k)\times
[0,\rho_{k+1})\times\cdots\times[0,\rho_{n})\subset\mathbb{R}^k\times\mathbb{R}_+^{n-k}$, called the {\em sum of the power series $s$}. Moreover, $f_s$ is real analytic at any point in the interior of $P_{k,n-k}^\rho$ and its germ at $\mathbf{0}\in\mathbb{R}^n$ is univocally determined by the series $s$. We define the \emph{convergent mixed power series} to be the elements of $\mathbb{R}\{Y,Z^*\}=\mathbb{R}[[Y,Z^*]]\cap\mathbb{R}\{(Y,Z)^*\}$.
%

\subsection{Standard and Generalized Analytic Manifolds}
%

Let $V$ be an open subset of $\mathbb{R}^n_{+}$ and let $g:V\to\mathbb{R}$ be a continuous function. Given a point $p=(p_1,p_2,\ldots,p_n)\in V$, consider $I_p=\{i:\,p_i=0\} \subset \{1,2,\ldots,n\}$, and put $\ell=\#I_p$ and $k=n-\ell$. We say that $g$ is \emph{generalized analytic} (or just \emph{$\mathcal{G}$-analytic}) at $p$ if there exists $s\in\mathbb{R}\{Y,Z^*\}$, where $Y$ is a $k$-tuple and $Z$ is an $\ell$-tuple, such that for any $x=(x_1,x_2,\ldots,x_n)$ in a sufficiently small neighbourhood of $\mathbf{0}\in\mathbb{R}^k \times \mathbb{R}_{+}^\ell$, we have
$$
g(p_1+x_1,p_2+x_2,...,p_n+x_n)=f_s(x_{\sigma(1)},x_{\sigma(2)},\ldots,x_{\sigma(n)}),
$$
where $\sigma$ is a permutation of the set $\{1,2,\ldots,n\}$ such that $\sigma(\{k+1,k+2,\ldots,n\})=I_p$.
We say that $g$ is {\em  generalized analytic in $V$} if it so at every point $p$ in $V$. In the definition above, the series $s$ is univocally determined by the 
germ of $g$ at $p$, up to permutation of the variables $Y$ and $Z$ separately. Thus, the set of germs of generalized analytic functions at $p$ defines an $\mathbb{R}$-algebra isomorphic to $\mathbb{R}\{Y,Z^*\}$. On the other hand, if $g$ is a generalized analytic function at some point $p \in \mathbb{R}^n_{+}$, then it is so in a neighbourhood of $p$ in $\mathbb{R}^n_{+}$. 
Summarizing, the assignment $\mathcal{G}_n:V\mapsto\mathcal{G}_n(V)$, where $V$ is an open subset of $\mathbb{R}^n_{+}$ and $\mathcal{G}_n(V)$ is the set of  generalized analytic functions in $V$, is a sheaf of $\mathbb{R}$-algebras of continuous functions over $\mathbb{R}^n_{+}$, where the stalks $\mathcal{G}_{n,p}$ are local algebras. 
Moreover $\mathcal{G}_n$ contains the sheaf $\mathcal{O}_n$ of analytic functions, where $\mathcal{O}_n(V)$ is the $\mathbb{R}$-algebra of real functions in $V$ which extend to real analytic functions on some open neighborhood of $V$ in $\mathbb{R}^n$.

\strut

With this formalism, and taking as local models the locally ringed spaces $\mathbb{O}_n:=(\mathbb{R}^n_{+},\mathcal{O}_n)$ and $\mathbb{G}_n:=(\mathbb{R}^n_{+},\mathcal{G}_n)$, we define both, the categories of \emph{standard and generalized (real) analytic manifolds (with boundary and corners)}. The objects in these categories are called \emph{$\mathcal{O}$-manifolds} and \emph{$\mathcal{G}$-manifolds}, respectively. In order to treat both together we write $\mathcal{A}$ to make reference either to $\mathcal{O}$ or to $\mathcal{G}$, and $\mathbb{A}$ to refer either to $\mathbb{O}$ or $\mathbb{G}$. An $\mathcal{A}$-manifold of dimension $n$ is a locally ringed space $\mathcal{M}=(M,\mathcal{A}_M)$, where $M$ is a second contable Hausdorff topological space (the {\em underlying space}),  $\mathcal{A}_M$ is a subsheaf of the sheaf $\mathcal{C}^0_M$ of germs of continuous real functions on $M$ (the {\em structural sheaf}), and it is locally isomorphic to the local model $\mathbb{A}_n$. That is, given $p\in M$ there is an open neighborhood $V$ of $p$ in $M$, an open subset $U$ of $\mathbb{R}^n_{+}$ and a homeomorphism $\varphi:V\to U$ inducing an isomorphism of the locally ringed spaces 
	$$
	(\varphi,\varphi^\#):(V,\mathcal{A}_{M}|_{V})\stackrel{\sim}{\longrightarrow}(U,\mathcal{A}_{n}|_{U}),
	$$
where $\varphi^\#_{p}:\mathcal{A}_{n,\varphi(p)}\to \mathcal{A}_{M,p}$ is given by the germ at $p$ of the composition $g \mapsto g \circ \varphi$.
A morphism between two $\mathcal{A}$-manifolds is just a morphism as locally ringed spaces, 
induced by composition with continuous maps on the underlying spaces (with an abuse of language, we frequently identify morphisms with the corresponding continuous maps). 
A couple $(V,\varphi)$ in the above conditions is called a {\em local chart} of $\mathcal{M}$ at $p$, the components $\mathbf{x}=(x_1,x_2,\ldots,x_n)$ of the isomorphism $\varphi:V \to U$ are {\em local coordinates at $p$}, and a family of local charts $\{(V_j,\varphi_j)\}_{j\in J}$ such that $M=\cup_{j\in J}V_j$ is an {\em atlas} of $\mathcal{M}$.

\strut

Let $\mathcal{M}=(M,\mathcal{A}_M)$ be an $\mathcal{A}$-manifold. Note that the underlying space $M$ is a topological manifold with boundary, denoted by $\partial M$, and that the restriction $(M\setminus\partial M,\mathcal{A}_M|_{M\setminus\partial M})$ is a standard analytic manifold without boundary (consequently, generalized analytic manifolds without boundary are also standard). Also there is a natural stratification $\mathcal{S}_{\mathcal{M}}$ of $\mathcal{M}$ described as follows. If $p\in M$, and $(V,\varphi)$ is a local chart at $p$, the number $e_p$ of vanishing coordinates in $\varphi(p)$ (equal to $\# I_{\varphi(p)}$) does not depend on the local chart $(V,\varphi)$ chosen (see \cite{Mar-Rol-San}). 
In that way, there is a well-defined map
$$
e:M\to\{0,1,\ldots,n\}, \quad\, p \mapsto e_p,
$$
which is upper semi-continuous. The elements of $\mathcal{S}_{\mathcal{M}}$ are the connected components of the fibers of $e$. Given $S \in \mathcal{S}_{\mathcal{M}}$, let us write $e_S=e(p)$, where $p$ is any point in $S$. Observe that $(S,\mathcal{A}_\mathcal{M}|_{S})$ is a standard analytic manifold of dimension $e_S$. In particular, the boundary $\partial M$ corresponds exactly with the points $p \in M$ with $e_p >0$, that is, $\partial M$ is equal to the union of strata of dimension strictly smaller than $n$. We have also that, $\partial M$ is a \emph{normal crossings divisor} with respect to the structural sheaf. That is, for each $p\in \partial M$, there exists a local chart $(V,\varphi)$ of $\mathcal{M}$ at $p$ such that
$$
\partial M \cap V =\{q\in V: \; x_1(q)\cdot x_2(q)\cdots x_{e_p}(q)=0\},
$$
where $(x_1,x_2,\ldots,x_n)$ are the coordinates associated to $\varphi$.

\begin{ejem} \label{ejem:traslacion}
	Let $\bar{\mathcal{O}}_k$ the sheaf of real (standard) analytic functions in $\mathbb{R}^k$. The locally ringed space $(\mathbb{R}^k,\bar{\mathcal{O}}_k)$ is in a natural way a generalized and standard analytic manifold, by means of the homeomorphism $\psi_k:\mathbb{R}^k \to (0, \infty)^k$ given by $(a_1,a_2,\ldots,a_k) \mapsto (e^{a_1},e^{a_2},\cdots,e^{a_k})$. 
\end{ejem}

We observe at this point that the product is defined in the category of $\mathcal{A}$-manifolds. That is, given two generalized or standard analytic manifolds $\mathcal{M}_1=(M_1,\mathcal{A}_{M_1})$ and $\mathcal{M}_2=(M_2,\mathcal{A}_{M_2})$ of dimensions $n$ and $m$, respectively, there is a natural $\mathcal{A}$-manifold of dimension  $n+m$, that we denote by $\mathcal{M}_1 \times \mathcal{M}_2=(M_1\times M_2,\mathcal{A}_{{M_1}\times{M}_2})$, unique up to isomorphism, solving the ``product universal property''. Without too much detail, the sheaf $\mathcal{A}_{{M_1}\times{M}_2}$ is recovered as follows. Given a point $(p,q)\in M_1\times M_2$ and two coordinate charts $\varphi_1:V_1\to U_1$ and $\varphi_2:V_2\to U_2$ at $p$ and $q$ respectively, we have that 
$$
\mathcal{A}_{{M_1\times M_2},(p,q)}=\{f\circ (\varphi_1\times \varphi_2)_{(p',q')}: \; f \in \mathcal{A}_{n+m,(p',q')}\},
$$
where $(p',q')=(\varphi_1(p),\varphi_2(q))$.
\begin{ejem}  \label{ejem:estructuramixta}
The product $(\mathbb{R}^k,\bar{\mathcal{O}}_k) \times (\mathbb{R}_+^{n-k},\mathcal{A}_{n-k})$, where $\mathcal{A} \in \{\mathcal{O},\mathcal{G}\}$, has a natural structure of $\mathcal{A}$-manifold by means of the homeomorphism $\psi_{k}\times \text{id}$, where $\psi_k$ has been introduced in Example \ref{ejem:traslacion}. We refer to this product by writing $(\mathbb{R}^k \times \mathbb{R}_+^{n-k}, \mathcal{A}_{k,n-k})$.
\end{ejem}
\begin{rem}  \label{rem:coordmixtas}
Let us consider a point $p\in M$ with $e_p=k$ and let $(V,\varphi)$ be a local chart of $\mathcal{M}$ at $p$. Up to permutation, we can assume that $\varphi(p)=(a_1,a_2,\ldots,a_k,0,\ldots,0)$ with $a_i \ne 0$ for all $i\in \{1,2,\ldots,k\}$. We can split the local coordinates $\mathbf{x}$ defined by $\varphi$ in two groups $\mathbf{x}=(\mathbf{y},\mathbf{z})$, where $\mathbf{y}=(y_1,y_2,\ldots,y_k)$ are standard analytic functions at $p$ and $\mathbf{z}=(z_{k+1},z_{k+2},\ldots,z_n)$ are generalized coordinate functions. By means of translations $y_i'=y_i-a_i$ in the analytic coordinates we obtain a new isomorphism
$$
\varphi': V \mapsto (\psi_{k}\times \text{id})^{-1}(\varphi(V)) \subset  \mathbb{R}^k \times \mathbb{R}_{+}^{n-k}.
$$
We consider also $\varphi'$ as a \emph{coordinate chart centered at $p$} in the sense that $\varphi'(p)=\mathbf{0} \in \mathbb{R}^k \times \mathbb{R}_+^{n-k}$, and we usually assume that our charts are centered charts.
\end{rem}

Let us recall now the expression in coordinates of the continuous maps inducing morphisms of generalized functions (details in \cite[Prop 3.16]{Mar-Rol-San}). Consider two generalized analytic manifolds $\mathcal{M}_1=(M_1,\mathcal{G}_{M_1})$ and $\mathcal{M}_2=(M_2,\mathcal{G}_{M_2})$ and a continuous function $\phi:M_1 \to M_2$ inducing a morphism between $\mathcal{M}_1$ and $\mathcal{M}_2$. Given $p\in M_1$ and $q=\phi(p)\in M_2$, take $(V_p,\varphi_p)$, $(W_q,\psi_q)$ charts centered at $p$ and $q$, respectively. Following notation in Remark \ref{rem:coordmixtas}, denote by $(\mathbf{y},\mathbf{z})$ the $k$ standard and $n-k$ generalized coordinates defining $\varphi_p$. Up to permutation, we can assume also that the first $k'$ coordinates defining $\psi_q$ are standard and the other $n'-k'$ are generalized. Then, the $j$-th component $\tilde{\phi}_j$ of $\tilde{\phi}=\psi_q \circ \phi \circ \varphi_p^{-1}$ is a generalized analytic function and for $j={k'+1},k'+2,\ldots,n'$, we have that
\begin{equation} \label{eq:morfimoslocal}
	\tilde{\phi}_j=\mathbf{z}^{\lambda_j}U_j(\mathbf{y},\mathbf{z}), \quad U_j(\mathbf{0},\mathbf{0})\ne 0, \quad \lambda_j \in \mathbb{R}_{+}^{n-k} \setminus \{0\}.
\end{equation}
Moreover, if $\phi$ induces an isomorphism, we have that $\phi$ is a homeomorphism, $n=n'$, $k=k'$, the map $\mathbf{t}\in\mathbb{R}^k\mapsto (\tilde{\phi}_1(\mathbf{t},0),\tilde{\phi}_2(\mathbf{t},0),\ldots,\tilde{\phi}_k(\mathbf{t},0))$ is an analytic isomorphism, and, if we write $\lambda_j=(\lambda_{j,1},\lambda_{j,2},\ldots,\lambda_{j,n-k})$ in Equation \eqref{eq:morfimoslocal}, up to a permutation of coordinates $\mathbf{z}$ we have
\begin{equation} \label{eq:isomorfimoslocal}
	\lambda_{j,\ell}=0, \quad \ell \in \{1,2,\ldots,n-k\}\setminus\{j-k\},
\end{equation}
for all $j={k+1},k+2,\ldots,n$.

\strut

We end this section introducing some notations and definitions about the strata of the natural stratification $\mathcal{S}_{\mathcal{M}}$. Given a stratum $S$ in $\mathcal{S}_{\mathcal{M}}$, denote by $\overline{S}$ the closure of $S$ in $M$, and define $\dim (\bar{S}):=\dim (S)$. We write $\mathcal{Z}_{\mathcal{M}}:=\{ \overline{S} \subset M: \; S \in \mathcal{S}_{\mathcal{M}}\}$. For $j=0,1,\ldots,n$, denote by $\mathcal{Z}^j_{\mathcal{M}}$ the set of elements in $\mathcal{Z}_{\mathcal{M}}$ with codimension $j$, that is 
$$
\mathcal{Z}^j_{\mathcal{M}}=\{\bar{S} \in \mathcal{Z}_{\mathcal{M}}: \; e_S=n-j\}
$$
The elements of $\mathcal{Z}_{\mathcal{M}}^0$ coincide with the strata of dimension 0 and are called {\em corner points}, the elements of $\mathcal{Z}_{\mathcal{M}}^1$ are called {\em edges} and the elements of $\mathcal{Z}_{\mathcal{M}}^{n-1}$ are called {\em components} of $\partial M$. Note that $\partial M$ is the union of its components. 

For each $Z \in \mathcal{Z}_{\mathcal{M}}$, we denote by $\mathcal{Z}_{\mathcal{M}}(Z)$ to the subset of $\mathcal{Z}_{\mathcal{M}}$ whose elements are contained on $Z$, and we write $\mathcal{Z}^j_{M}(Z)=\mathcal{Z}_{\mathcal{M}}(Z) \cap \mathcal{Z}_{\mathcal{M}}^j$, for each $j=0,1,\ldots,n$. We usually write for short $p\in \mathcal{Z}_{\mathcal{M}}^0$ instead of $\{p\}\in \mathcal{Z}_{\mathcal{M}}^0$, and when no confusion arises, we will write $\mathcal{Z}$ instead of $\mathcal{Z}_\mathcal{M}$, $\mathcal{Z}^j$ instead of $\mathcal{Z}_\mathcal{M}^j$, etc.

\subsection{Monomial Complexity along Strata}

We introduce in this section the concept of monomial complexity along a stratum and the definition of stratified monomial type function.

\strut

Let us consider a generalized analytic manifold $\mathcal{M}=(M,\mathcal{G}_M)$ and a stratum $S$ of its natural stratification $\mathcal{S}_{\mathcal{M}}$. Take a local chart $(V,\varphi)$ of $\mathcal{M}$ centered at $p\in S$, write $e=e_S$ and $k=\dim S=n-e$.
We can split the coordinates defining $\varphi$, up to reorder them, as $(\mathbf{y}, \mathbf{z})$, where $\mathbf{y}=(y_{1},y_{2},\ldots,y_{k})$ are standard analytic coordinates in $S \cap V$ and $\mathbf{z}=(z_1,z_2,\ldots,z_e)$ are generalized functions such that 
$
S \cap V=\{q \in V: \;z_1(q)=z_2(q)=\cdots=z_e(q)=0\}.
$
Shrinking $V$ if it is necessary, this chart $\varphi$ provides an isomorphism
$$
\Psi_\varphi^S: \mathbb{R}\{Y, Z^*\} \to \mathcal{G}_{M}(V), \quad s\mapsto f_s\circ \varphi
$$
where $Y$ and $Z$ are $k$ and $e$ tuples, respectively, and $f_s$ is the sum of the power series $s$ that has been introduced in Equation \eqref{eq:serieconvergente}.
Given $f\in \mathcal{G}_{M}(V)$ and $s\in \mathbb{R}\{Y,Z^*\}$ the mixed power series such that $\Psi^S_\varphi(s)=f$, we denote
\begin{equation} \label{eq:soporteS}
	\text{Supp}_S(f;\varphi)=\text{Supp}(s^Z) \subset \mathbb{R}^{e}_+,\quad \text{Supp}_{\min,S}(f;\varphi)=\text{Supp}_{\min}(s^Z) \subset \mathbb{R}^{e}_+,
\end{equation}
where $s^Z\in \mathbb{R}\{Y\}\{Z^*\}$ has been introduced in Equation \eqref{eq:monomorfismo}.

\begin{lema} \label{lema:soporteminimocambiocarta}
Let $S$ be a stratum in $\mathcal{S}_\mathcal{M}$ with $e=e_S$. Take an open subset $U$ of $M$ such that $U \cap S \ne \emptyset$ and a function $f\in \mathcal{G}_M(U)$. Let us consider two local charts $(V_1,\varphi_1)$, $(V_2,\varphi_2)$ centered at $p$ and $q$ respectively, with $p,q\in S \cap U$. There exists a vector $(\gamma_{1},\gamma_{2},\ldots,\gamma_e)\in \mathbb{R}^{e}_{>0}$ such that $(\lambda_1,\lambda_2,\ldots,\lambda_e) \in \text{Supp}_{\min,S}(f;\varphi_2)$ if and only if $(\gamma_{1}\lambda_{1},\gamma_{2}\lambda_{2},\ldots,\gamma_e\lambda_e) \in\text{Supp}_{\min,S}(f;\varphi_1)$. 
\end{lema}
\begin{proof}
Using that $S$ is path connected and by compactness of a given path from $p$ to $q$, we can reduce ourselves to the case where both points $p$ and $q$ belong to the same connected component $W$ of $U\cap V_1\cap V_2$. Write $\mathbf{y}=(y_1,y_2,\ldots,y_{n-e})$, $\mathbf{z}=(z_1,z_2,\ldots,z_{e})$, and $\bar{\mathbf{y}}=(\bar{y}_1,\bar{y}_2,\ldots,\bar{y}_{n-e})$, $\bar{\mathbf{z}}=(\bar{z}_1,\bar{z}_2,\ldots,\bar{z}_{e})$, where, up to reordering, $(\mathbf{y},\mathbf{z})$ are the coordinate functions associated to $\varphi_1$ and $(\bar{\mathbf{y}},\bar{\mathbf{z}})$ are the ones associated to $\varphi_2$, in such a way that $\mathbf{y}|_{S\cap V_1}$,  $\bar{\mathbf{y}}|_{S\cap V_2}$ are analytic coordinates in $W \cap S$. That is, we have that 
$$
W \cap S = \{z_1=z_2=\cdots=z_e=0\}=\{\bar{z}_1=\bar{z}_2=\cdots=\bar{z}_e=0\}.
$$
Up to a new reorder of the variables, for instance  of $\mathbf{z}$, the change of coordinates $\varphi_2 \circ \varphi_1^{-1}$ satisfies that $\bar{y}_j=g_j(\mathbf{y},\mathbf{z})$, and $\bar{z}_\ell=z_\ell^{\gamma_\ell}h_\ell(\mathbf{y},\mathbf{z})$, where $\gamma_\ell > 0$ and $h_\ell(\mathbf{0},\mathbf{0})\ne 0$, for any $\ell=1,2,\ldots,n-k$, in view of Equations \eqref{eq:morfimoslocal} and \eqref{eq:isomorfimoslocal}. We summarize these expressions by writing $\bar{\mathbf{y}}=g$ and $\bar{\mathbf{z}}=\mathbf{z}^\gamma h$. If $\Delta_2=\text{Supp}_{\min,S}(f;\varphi_2)=\{\mu_1,\mu_2,\ldots,\mu_t\}$, the expression of $f$ in coordinates $(\bar{\mathbf{y}},\bar{\mathbf{z}})$ is 
$$
\begin{array}{cc}
	f|_W= \bar{\mathbf{z}}^{\mu_1}V_1(\bar{\mathbf{y}},\bar{\mathbf{z}})+ \bar{\mathbf{z}}^{\mu_2}V_2(\bar{\mathbf{y}},\bar{\mathbf{z}})+\cdots+ \bar{\mathbf{z}}^{\mu_t}V_t(\bar{\mathbf{y}},\bar{\mathbf{z}}),
\end{array}
$$
where $V_j(\bar{\mathbf{y}},\mathbf{0}) \not\equiv 0$, for any $j=1,2,\ldots,t$. 
Applying the change of coordinates in order to get the expression of $f$ in $(\mathbf{y},\mathbf{z})$, we obtain
$$
f|_W= \mathbf{z}^{\gamma\mu_1}W_1(\mathbf{y},\mathbf{z})+\mathbf{z}^{\gamma\mu_2}W_2(\mathbf{y},\mathbf{z})+\cdots+ \mathbf{z}^{\gamma\mu_t}W_1(\mathbf{y},\mathbf{z}), \quad W_j(\mathbf{y},\mathbf{z})=h^{\mu_j}V_j(g,\mathbf{z}^\gamma h),
$$
where $\gamma \mu_j=(\gamma_1\mu_{j,1},\gamma_2\mu_{j,2},\ldots,\gamma_{e}\mu_{j,e})$. Note that $W_j(\mathbf{y},\mathbf{0}) \ne 0$, for any $j=1,2,\ldots,t$. As a consequence, each $\lambda \in \Delta_1=\text{Supp}_{\min,S}(f;\varphi_1)$ is such that $\gamma\mu_j \leq_d \lambda$, for some $j\in \{1,2,\ldots,t\}$; hence $\Delta_1  \subset \bar{\Delta}_2=\{\gamma \mu_1,\gamma \mu_2,\ldots, \gamma \mu_t\}$. Now, recall that, any pair of elements $\mu_j,\mu_k \in \Delta_2$ are incomparable for the order $\leq_d$. Then, the elements of $\bar{\Delta}_2$ are also mutually incomparable and the equality $\Delta_1=\bar{\Delta}_2$ holds.
\end{proof}

Next definition makes sense as a result of Lemma \ref{lema:soporteminimocambiocarta}.
\begin{defin}
Let $f:M \to \mathbb{R}$ be a generalized analytic function. The \emph{monomial complexity $m_S(f)$ of $f $ along $S$} is the number $m_S(f)=\#\text{Supp}_{\min}(f|_V,S; \varphi,p)$, where $(V,\varphi)$ is a local chart centered at any point $p \in S$.
\end{defin}
When $S=\{p\}$, we just write $m_p(f)$ instead of $m_{\{p\}}(f)$. 

\begin{lema}[Horizontal stability] \label{lema:estabilidadhorizontal}
	Let $f:M \to \mathbb{R}$ be a generalized analytic function. Given two strata $S$ and $T$ such that $T \subset \bar{S}$, we have the inequality $m_S(f) \leq m_T(f)$.
\end{lema}
\begin{proof}
	It is a direct consequence of Equation \eqref{eq:soporteminimoproyeccion}.
\end{proof}	


\begin{defin}
	Let $\mathcal{M}=(M,\mathcal{G}_M)$ be a generalized analytic manifold. A generalized analytic function $f:M\to\mathbb{R}$ is of {\em stratified monomial type} if $m_S(f)=1$ for any stratum $S\in \mathcal{S}_M$.
\end{defin}

\subsection{Standardizations and Blowing-ups}
Let $\mathcal{M}=(M,\mathcal{A}_M)$ be a standard or generalized analytic manifold and let $Y \subset M$ be a connected closed subset of $M$. We say that $Y$ is a {\em geometrical center for $\mathcal{M}$} if at each $p \in Y$ there is a local chart $(V,\varphi)$ centered at $p$ and some $r\in\{1,2,\ldots,n\}$, such that
$$
Y\cap V=\{q \in V: \; x_1(q)=x_2(q)=\cdots=x_r(q)=0\},
$$ 
where $(x_1,x_2,\ldots,x_n)$ are the coordinates defined by $\varphi$. For instance, if $Z\in\mathcal{Z}^j$ for some $j \leq n-1$, then $Z$ is a geometrical center (the number $r$ in the definition above is $n-j$); they are called \emph{combinatorial geometrical centers}. 

When $\mathcal{M}=(M,\mathcal{O}_M)$ is a standard analytic manifold (with boundary), the construction of the {\em (real) blowing-up with center $Y$} is quite well-known (see details in \cite{Mar-Rol-San}). It consists of a proper morphism of standard analytic manifolds
$$
(\pi_Y,\pi_Y^{\#}):(\widetilde{M},\mathcal{O}_{\widetilde{M}})\to(M,\mathcal{O}_M)
$$
inducing an isomorphism between $\widetilde{M}\setminus E$ and  $M\setminus Y$, where the \emph{exceptional divisor} $E=\pi_Y^{-1}(Y)$ is a new component of $\partial\widetilde{M}$. On the contrary, when $\mathcal{M}=(M,\mathcal{G}_M)$ is a generalized analytic manifold, the blowing-up of $\mathcal{M}$ with geometrical center $Y$ may even not exist and, if it does, depends on the so called {\em standardization} of $\mathcal{M}$ (see  \cite{Mar-Rol-San}). We devote this section to recall this concept and what do we mean by blowing-up in the category of generalized analytic manifolds.

\strut 

Let $\mathcal{N}=(N,\mathcal{O}_{N})$ be a standard analytic manifold. Given $p\in N$ and $\varphi_p:V_p \to U_p$ a local chart at $p$, we consider $\mathcal{O}^\epsilon_N(V_p)$ to be the $\mathbb{R}$-algebra of continuous functions $f:V_p\to \mathbb{R}$ such that $f \circ \varphi_p^{-1}$ belongs to $\mathcal{G}_n(U_p)$. Taking the sheaf associated to the presheaf obtained by means of these assignations, we obtain a generalized analytic manifold $\mathcal{N}^\epsilon=(N, \mathcal{O}^\epsilon_N)$, called the \emph{enrichment of $\mathcal{N}$} (see \cite[Prop. 3.17]{Mar-Rol-San}). Note that the stratifications $\mathcal{S}_{\mathcal{N}^\epsilon}$ and $\mathcal{S}_\mathcal{N}$ coincide. Moreover, if $Y$ is a geometrical center for $\mathcal{N}$, then it is also a geometrical center for $\mathcal{N}^\epsilon$.

The assignment $\mathcal{N} \to \mathcal{N}^\epsilon$ is not a functor from the category of $\mathcal{O}$-manifolds to the one of $\mathcal{G}$-manifolds, as we have already seen that the morphisms do not lift to the enrichments unless they can be expressed locally as tuples of monomial-type functions.


\begin{defin}
A \emph{standardization of a generalized analytic manifold $\mathcal{M}=(M,\mathcal{G}_M)$} is a subsheaf $\mathcal{O}$ of $\mathcal{G}_M$ such that $\mathcal{N}=(M,\mathcal{O})$ is a standard analytic manifold with $\mathcal{N}^\epsilon=\mathcal{M}$. A  generalized analytic manifold $\mathcal{M}$ is said to be {\em standardizable} if there exists a standardization $\mathcal{O}\subset \mathcal{G}_M$ of it. 
\end{defin}
Note that an standardization is the same thing as providing an atlas $\mathfrak{A}_M=\{(V_j,\varphi_j)\}_{j\in J}$ of $\mathcal{M}$ such that, for any $i,j \in J$, the changes of coordinates $\varphi_j\circ\varphi_i^{-1}$ are standard analytic in their domains of definition $\varphi_i(V_i\cap V_j)$.

\begin{rem}
Let $\mathcal{M}=(M,\mathcal{A})$ be an analytic manifold without boundary (hence $\mathcal{M}$ is at the same time standard and generalized). Take a point $p\in M$, an open neigbourhood $V$ of $p$, and a coordinate system $(x_1,x_2,\ldots,x_n)$ defined in $V$ and centered at $p$. For each $i=1,2,\ldots,n$, take positive integers $m_i \in \mathbb{Z}_{>0}$ and the functions $y_i=x_i^{m_i}$. The map $\varphi:V \to \mathbb{R}^n$ defined by $\varphi(q)=(y_1(q),y_2(q),\ldots,y_n(q))$ is an homeomorphism onto $U=\varphi(V)$ and the sheaf $\tilde{\mathcal{O}}$ defined locally at $p\in V$ by
$$
\tilde{\mathcal{O}}_p=\{f\circ \varphi_p: \; f \in \bar{\mathcal{O}}_{n,\varphi(p)}\}\subset \mathcal{A}_p	
	$$
is a subsheaf of $\mathcal{A}|_V$ so that $(V,\tilde{\mathcal{O}})$ is a standard analytic manifold. However $\tilde{\mathcal{O}}^\epsilon=\mathcal{A}|_V$ if and only if $m_i=1$ for all $i=1,2,\ldots,n$, or equivalently, if $\varphi$ is a local chart if $\mathcal{M}$. In this case, there is a unique standardization for $\mathcal{M}$: the total sheaf $\mathcal{A}$ itself.

On the contrary, when $\mathcal{M}=(M,\mathcal{G})$ is a generalized analytic manifold with $\partial M \ne \emptyset$, we may have a lot of variation. For example if we take a point $p\in \partial M$
and a small enough neighbourhood $V$ of $p$, there are infinitely many standardizations of the local generalized analytic manifold $(V,\mathcal{G}|_V)$. But there are also examples of non-standardizable generalized manifolds like the one in \cite[Example 3.20]{Mar-Rol-San}.
\end{rem}

\begin{defin}
A {\em center $\xi$ of blowing-up for a generalized analytic manifold $\mathcal{M}$} is a pair $\xi=(Y,\mathcal{O})$, where $\mathcal{O}$ is an standardization of $\mathcal{M}$ and $Y$ is a geometrical center for $(M,\mathcal{O})$.
\end{defin}

\begin{rem}
	We can have a geometrical center $Y$ for a generalized manifold $\mathcal{M}$ and an standar\-dization $\mathcal{O}$ of $\mathcal{M}$ such that $Y$ is not a geometrical center for $\mathcal{N}=(M,\mathcal{O})$. For example, let us take the generalized analytic manifold $\mathcal{M}=(V,\mathcal{G}_{1,1}|_V)$, where $V \subset\mathbb{R}\times \mathbb{R}_+$ is a small neighbourhood of the origin $(0,0)\in \mathbb{R}^2$. Let $(y,z)$ be the natural coordinates in $\mathbb{R}^2$ and let $Y$ be the closed topological subspace of $V$ given by the zeros of $y-z^\lambda$, where $\lambda \notin \mathbb{Z}_{>0}$. Note that $(y',z)$ with $y'=y-z^\lambda$ are also coordinates of $\mathcal{M}$ at the origin, then $Y$ is a geometrical center for $\mathcal{M}$. But, if we take the standardization $\mathcal{O}\subset \mathcal{G}_{1,1}|_V$ given by the local chart $(y,z)$, then $Y$ is not a geometrical center for $\mathcal{N}=(V,\mathcal{O})$.
\end{rem}

\strut

Now we have the ingredients to introduce the blowing-up morphisms in the category of generalized manifolds. 

\begin{defin}
Let $\mathcal{M}$ be a generalized analytic manifold and let $\xi=(Y,\mathcal{O})$ be a center of blowing-up for $\mathcal{M}$. The {\em blowing-up $\pi_\xi: \mathcal{M}_\xi \to \mathcal{M}$ with center $\xi$} 
is the morphism of $\mathcal{G}$-manifolds induced by the blowing-up $\pi_Y:\widetilde{\mathcal{N}}\to\mathcal{N}$ of the standard manifold $\mathcal{N}=(M,\mathcal{O})$ with center $Y$ (note that we are writing $\mathcal{M}_\xi = \tilde{\mathcal{N}}^e$).
\end{defin}

If $Z \in \mathcal{Z}_\mathcal{M}$ is a combinatorial geometrical center, and $\mathcal{O} \subset \mathcal{G}_M$ is a standardization, we can see that $\xi=(Z,\mathcal{O})$ is a center of blowing-up for $\mathcal{M}$. Such a $\xi$ is called a \emph{combinatorial center of blowing-up}, and we say that $\pi_\xi:\mathcal{M}_\xi \to \mathcal{M}$ is a \emph{combinatorial blowing-up}.

%

\section{The Category of Monomial Analytic Manifolds}

In this section we introduce a subcategory of generalized analytic manifolds, called \emph{monomial generalized analytic manifolds}, which has many combinatorial properties. The objects of this subcategory are those $\mathcal{G}$-manifolds equipped with an atlas for which the change of coordinates are expressed as monomial morphisms. We codify these changes of coordinates by means of matrices of exponents, and also we do the same for the morphisms. The formulation in this combinatorial language allow us to conclude that monomial manifolds are always standardizable. This result will be the key to prove the stratified reduction of singularities.

\subsection{Monomial Manifolds}

We devote this section to introduce the category of \emph{monomial generalized or standard analytic manifolds}. 

\strut

We consider an $\mathcal{A}$-manifold $\mathcal{M}=(M,\mathcal{A}_M)$, where $\mathcal{A}\in \{\mathcal{O},\mathcal{G}\}$. Let us write $\partial M= \cup_{i\in I} E_i$, where $\{E_i\}_{i \in I}$ is the family of components of $\partial M$, and assume that $I$ is a finite set. As in \cite{Mol}, we say that $\partial M$ is a \emph{strong normal crossing divisor} if for each $J \subset I$ the intersection $E_J= \cap_{j\in J} E_i$ is a connected set (in particular $E_\emptyset=M$ is connected). In this case, we have natural bijections $\mathcal{H}\to \mathcal{S} \to \mathcal{Z}$, where
$$
\mathcal{H}=\mathcal{H}_{\mathcal{M}}=\{J \subset I ;  \; E_J\ne \emptyset\},
$$
given by $J \mapsto E_J \mapsto S_J$, where $S_J \in \mathcal{S}$ is the stratum of $M$ such that $E_J= \overline{S}_J$. Given $Z \in \mathcal{Z}$, the element $I_Z \in \mathcal{H}$ such that $E_{I_Z}=Z$ is called the \emph{index set of $Z$}. We use the notation $I_p:=I_{\{p\}}$ when $p\in \mathcal{Z}^0$. Observe that $\# I_Z =\dim M-\dim Z$. In particular, if $p\in \mathcal{Z}^0$ we have that $I_p$ is a set with $n$ elements. Note also that $\tilde{Z} \in \mathcal{Z}(Z)$ if and only if $I_Z \subset I_{\tilde{Z}}$.

\begin{rem} 
Let $Y \in \mathcal{Z}^1$ be a compact edge. Under the condition of strong normal crossings for $\partial M$, we have that $\mathcal{Z}^0(Y)$ consists exactly on two corner points $p$ and $q$ and there are exactly two different elements $i_p,i_q \in I \setminus I_Y$ such that $I_p=I_Y \cup \{i_p\}$ and $I_q=I_Y \cup \{i_q\}$. 
\end{rem} 

Let us fix from now on an $\mathcal{A}$-manifold $\mathcal{M}=(M,\mathcal{A}_M)$ with at least one corner point and such that $\partial M$ is a strong normal crossings divisor. Given $p\in \mathcal{Z}^0$, define the set
$$
V^{\star}_p:=\cup\{ S_J\in \mathcal{S_M}: \; J \subset I_p\}= \cup \{S \in \mathcal{S_M}: \;  p\in \overline{S}\}.
$$
We have that $V^\star_p$ is an open neighbourhood of $p$ in $M$ homeomorphic to $\mathbb{R}^n_{+}$. Let us suppose that $\mathcal{M}$ has an \emph{affine chart} $(V^\star_p,\varphi_p)$ at $p$, that is, $\varphi_p$ is defined in $V^\star_p$, $\varphi_p(p)=0$ and  $\varphi_p(V^\star_p)=\mathbb{R}^n_{+}$. For any $i\in I_p$, we denote by $x_{p,i}:V^\star_p \to \mathbb{R}$ the coordinate component of $\varphi_p$ satisfying
$$
E_i \cap V^\star_p=\{q \in V^\star_p: \; x_{p,i}(q)=0\}.
$$
The family of functions $\mathbf{x}_p=(x_{p,i})_{i \in I_p}$ is equal to the family of coordinates of $\varphi_p$. Regardless of the ordering of these coordinates, we just identify $\varphi_p$ with $\mathbf{x}_p$. Even more, since the sets $V^\star_p$ are completely determined by the stratification of $\mathcal{M}$, we identify $(V^\star_p,\varphi_p)$ with $\mathbf{x}_p$ and we say simply that $\mathbf{x}_p$ is an affine chart over $V^\star_p$.

Assume that $\mathcal{M}$ has an atlas $\mathfrak{a}=\{\mathbf{x}_p\}_{p \in \mathcal{Z}^0}$, where each $\mathbf{x}_p$ is an affine chart over $V^\star_p$. With such a convention note that, in fact $\mathfrak{a}$ is not exactly an atlas, but an equivalence class of atlases, since we have not considered a particular ordering of the coordinates $\mathbf{x}_p$. We say that $\mathfrak{a}$ is a \emph{monomial atlas} if all the changes of coordinates have a purely monomial expression. More precisely, given two corner points $p,q \in \mathcal{Z}^0$, for each $i \in I_q$ there exist maps $r_i:I_p \to \mathbb{R}$, given by $j \mapsto r_i(j)=:r_{ij}$ such that the change of coordinates $\mathbf{x}_q \circ \mathbf{x}_p^{-1}$ has the following expression
\begin{equation}\label{eq:relacioncoordenadas}
x_{q,i}=\mathbf{x}_{p}^{r_i}, \quad  \mathbf{x}_{p}^{r_i}=\Pi_{j\in I_p}x_{p,j}^{r_{ij}}, \quad r_{ij}=r_i(j).
\end{equation}

\begin{defin}
A \emph{monomial $\mathcal{A}$-manifold} is a pair $(\mathcal{M},\mathfrak{a})$, where  $\mathcal{M}$ is a combinatorially stratified $\mathcal{A}$-manifold and  $\mathfrak{a}=\{\textbf{x}_p\}_{p\in \mathcal{Z}^0}$ is a monomial atlas over $\mathcal{M}$.
\end{defin}

\begin{rem} \label{rem:variosatlasmonomiales}
Assume that $\mathcal{M}$ is an $\mathcal{A}$-manifold admitting two monomial atlases $\mathfrak{a}$ and $\mathfrak{a}'$, that are compatible in the sense that for all $p \in \mathcal{Z}^0$ we have that $\mathbf{x}'_p\circ \mathbf{x}_p^{-1}$ has a purely monomial expression, where $\mathbf{x}_p\in \mathfrak{a}$ and $\mathbf{x}'_p\in \mathfrak{a}'$. When $\mathcal{A}=\mathcal{O}$ we have necessarily that $\mathfrak{a}=\mathfrak{a}'$. On the contrary, when $\mathcal{A}=\mathcal{G}$ we have a lot of variation. Indeed, from the monomial manifold $(\mathcal{M},\mathfrak{a})$, we can obtain many different monomial atlases for $\mathcal{M}$ just by changing at a single point $p\in \mathcal{Z}^0$ the affine coordinates $\mathbf{x}_p\in \mathfrak{a}$ over $V^\star_p$, by the affine coordinates $\mathbf{y}_p$ over $V^\star_p$ defined by $y_{p,i}=x_{p,i}^{s_i}$, with $s_i\in \mathbb{R}_{>0}$, for all $i\in I_p$.
\end{rem}

\begin{defin}
Let us consider two monomial $\mathcal{A}$-manifolds $(\mathcal{M}_1,\mathfrak{a}_1)$ and $(\mathcal{M}_2,\mathfrak{a}_2)$ and let $\phi:M_1 \to M_2$ be a continuous map. We say that $\phi$ is a \emph{morphism of monomial $\mathcal{A}$-manifolds} if $\phi$ provides an $\mathcal{A}$-manifolds morphism between $\mathcal{M}_1$ and $\mathcal{M}_2$, and moreover:

\begin{itemize}
\item For any point $p\in \mathcal{Z}^0_1$ we have that $\bar{p}=\phi(p)\in \mathcal{Z}^0_2$, where $\mathcal{Z}_1=\mathcal{Z}_{\mathcal{M}_1}$ and $\mathcal{Z}_2=\mathcal{Z}_{\mathcal{M}_2}$.
\item The expression of $\phi$ in the atlases $\mathfrak{a}_1$ and $\mathfrak{a}_2$ of $\mathcal{M}_1$ and $\mathcal{M}_2$, respectively, is monomial; that is, if $p\in \mathcal{Z}^0_1$ and $\bar{p}=\phi(p)$,  $\bar{\mathbf{x}}_{\bar{p}}  \in \mathfrak{a}_2$ and $\mathbf{x}_p \in \mathfrak{a}_1$, the composition $\bar{\mathbf{x}}_{\bar{p}} \circ \phi \circ \mathbf{x}_p^{-1}$ is written as 
\begin{equation} \label{eq:morfismos}
\bar{x}_{\bar{p},i}=\mathbf{x}_p^{b_i}, \; \text{ where } b_i:I_p  \to \mathbb{R}_+, \; \text{ for all } i \in I_{\bar{p}}.
\end{equation}
\end{itemize}
\end{defin}

The first example of monomial $\mathcal{A}$-manifold is an \emph{m-corner} $(\mathcal{M}_p, \{\mathbf{x}_p\})$ which consists on the following: Given an $\mathcal{A}$-manifold $\mathcal{M}=(M,\mathcal{G}_M)$ and an affine chart $(V,\varphi)$ centered at $p$, we take  $\mathcal{M}_p=(V,\mathcal{G}_M|_V)$ and $\mathbf{x}_p$ the tuple of components of $\varphi$. We write simply $(\mathcal{M}_p, \mathbf{x}_p)$.

\subsection{Combinatorial Data of Monomial Manifolds}

We codify the objects and morphism on the category of monomial $\mathcal{A}$-manifolds by means of their associated \emph{combinatorial data}. We need to introduce some notation:

\begin{notation}
Let $\mathbb{R}^{I}$ be the set of maps $I \to \mathbb{R}$, where $I$ is a finite set. 
\begin{itemize}
	\item We define $\mathds{1}_I : I \to \mathbb{R}$ to be the element of $\mathbb{R}^I$ such that $\mathds{1}_I(i)=1$, for all $i\in I$. 
	\item Given $\lambda \in \mathbb{R}^I$, we denote by $D_\lambda$ the element of $\mathbb{R}^{I \times I}$ such that $D_\lambda(i,i)=\lambda(i)$, for all $i \in I$ and $D_\lambda (i,j)= 0$, when $i\ne j$.
	\item Given  $A: {I\times J} \to \mathbb{R}$ and $B: {J\times K} \to \mathbb{R}$, with $I$, $J$ and $K$ finite sets, we define:
	\begin{itemize}
		\item $AB$ to be the element of $\mathbb{R}^{I\times K}$ given by $(i,k)\mapsto\sum_{j\in J} A(i,j)B(j,k)$. 
		\item $A^{-1}$ to be (if it exists) the element of $\mathbb{R}^{J\times I}$ such that $A^{-1}A= D_{\mathds{1}_{J}}$, and also $AA^{-1}= D_{\mathds{1}_{I}}$.
	\end{itemize}
\end{itemize}
Roughly, we consider maps $A: I \times J \to \mathbb{R}$ as matrices of size $\#I \times \#J$ with real coefficients, but without an specified order in the sets $I$ and $J$. We frequently use the matrix notation $A_{ij}=A(i,j)$, for $(i,j)\in I \times J$. 
\end{notation}

Take a monomial $\mathcal{A}$-manifold $(\mathcal{M},\mathfrak{a})$. Given two corner points $p,q \in\mathcal{Z}^0$, let $\mathbf{x}_p$ and $\mathbf{x}_q$ be the affine coordinates at $p$ and $q$, respectively, belonging to $\mathfrak{a}$. We codify the change of coordinates $\mathbf{x}_{q}\circ \mathbf{x}^{-1}_{p}$ expressed by the relations in Equation \eqref{eq:relacioncoordenadas} by means of the \emph{matrix of exponents} 
$$
C^{pq}:I_q\times I_p \to \mathbb{R}, \; \text{given by } (i,j) \mapsto r_{ij}.
$$
Note that we have the equality $C^{qp}=(C^{pq})^{-1}$.
\begin{defin}
The \emph{combinatorial data $\mathfrak{C}_{(\mathcal{M},\mathfrak{a})}$ of a monomial $\mathcal{A}$-manifold $(\mathcal{M},\mathfrak{a})$} is the collection of matrices of exponents $C^{pq}$, such that $p,q \in \mathcal{Z}^0$. 
\end{defin}

Let $p,q \in\mathcal{Z}^0$ be two corner points. A \emph{path $\mathcal{P}$ (of compact edges) from $p$ to $q$} is a list $\mathcal{P}=(Y_1,Y_2,\ldots,Y_k)$ where each $Y_j$ is a compact edge in $\mathcal{M}$, such that $p \in Y_1$, $q \in Y_k$, and the intersection $p_j =Y_j\cap Y_{j+1}$ is a corner point, for all $j=1,2,\ldots,k-1$. 
Note that, for any pair of corner points $p,q \in \mathcal{Z}^0$, there is always a path $\mathcal{P}=(Y_1,Y_2,\ldots,Y_k)$ from $p$ to $q$. Moreover, we can assume that $Y_1 \cup Y_2 \cup \cdots \cup Y_k \subset E_J$, where $J=I_p \cap I_q$, because of the connectedness of $E_J$. We say in this case that $\mathcal{P}$ is a \emph{path for $p$ to $q$ inside $E_J$}. Given a path $\mathcal{P}$ from $p$ to $q$ inside $E_J$, we have the equality
\begin{equation}\label{eq:bastanaristas}
	C^{pq}=C^{p_{k-1}q}\cdots C^{p_1p_2}C^{pp_1} \in \mathbb{R}^{I_q\times I_p},
\end{equation}
where $p_i=Y_i \cap Y_{i+1}$, for each $i=1,2,\ldots,k-1$. That is, the matrices of exponents between corner points connected by edges generate the whole combinatorial data $\mathfrak{C}_{(\mathcal{M},\mathfrak{a})}$ just by taking products. From now on, if $p$ and $q$ are the only two corner points in a compact edge $Y$, we simply say that $p$ and $q$ are connected through $Y$, and we write to emphasize $C_Y^{pq}=C^{pq}$. 

\strut

Let $(\mathcal{M},\mathfrak{a})$ be a monomial generalized analytic manifold. Let us show some properties for the combinatorial data of $(\mathcal{M},\mathfrak{a})$.

\strut

A corner point $p \in \mathcal{Z}^0$ belongs to $E_K=\overline{S}_K$ if and only if $K \subset I_p$. Hence, if we fix two corner points $p$ and $q$ in $\mathcal{Z}^0$, the smallest stratum containing both $p$ and $q$ in its adherence is $S_J$ with $J=I_p \cap I_q$. Moreover, we have 
$$
V^{\star}_p \cap V^{\star}_q = \bigcup_{K \subset J} S_K.
$$
Let $\mathbf{x}_p, \mathbf{x}_q\in \mathfrak{a}$ be the affine coordinates at $p$ and $q$, respectively. We consider the coordinate systems $\mathbf{x}^a_p=\{x_{p,i}^*\}_{i \in I_p}$ and $\mathbf{x}^a_q=\{x_{q,j}^*\}_{j \in I_q}$, centered at $a\in S_J$, defined by $x_{p,i}^*=x_{p,i}-x_{p,i}(a)$ and $x_{q,j}^*=x_{q,j}-x_{q,j}(a)$, for corresponding $i\in I_p$ and $j\in I_q$, respectively. Note that the coordinate functions $x^*_{p,i}$ (resp. $x^*_{q,j}$) are standard analytic at $a$ if and only if $i \in I_p\setminus J$ (resp. $j \in I_q\setminus J$), hence taking into account Equations \eqref{eq:morfimoslocal} and \eqref{eq:isomorfimoslocal} about the local expression of morphisms in arbitrary generalized analytic manifolds, for all $i,j \in J$, we get necessarily that
\begin{equation}\label{eq:partediagonal}
C^{pq}_{ij}=0, \; \text{ if } i \ne j; \quad C^{pq}_{ij} \in \mathbb{R}_{>0}, \; \text{ if } i=j.
\end{equation}
where $C^{pq} \in \mathfrak{C}_{(\mathcal{M},\mathfrak{a})}$ is the map of exponents codifying the change of coordinates $\mathbf{x}_q\circ \mathbf{x}_p^{-1}$.

In the following statement, we determine some other entries of the matrix of exponents $C^{pq}$ in the case where $p$ and $q$ are corner points connected through an edge.
\begin{lema} \label{lema:triangularinferior}
Let $Y$ be a compact edge in $M$ and let $\mathcal{Z}^0(Y)=\{p,q\}$. Denote by $i_p\in I_p$ and $i_q\in I_q$ to the indices such that $I_Y=I_q\setminus \{i_q\}=I_p\setminus \{i_p\}$. The map $C=C_Y^{pq}\in \mathfrak{C}_{(\mathcal{M},\mathfrak{a})}$ satisfies
$$
{\rm(a)}\; C_{ii}\in \mathbb{R}_{>0}, \quad {\rm(b)}\; C_{ij}=0,\, \text{ if } i\ne j, \quad {\rm(c)}\; C_{i_q,i}=0, \quad \text{ for all } i,j\in I_Y.
$$
\end{lema}
\begin{proof}
Assertions (a) and (b) are true in view of Equation \eqref{eq:partediagonal}. Let us prove (c). Let us suppose that there is an index $i \in I_Y$ such that $C_{i_q,i}\ne 0$ and let us find a contradiction. Denote by $S_i$ the stratum such that $\overline{S}_i=E_i$, that is
$$
S_i=E_i \setminus \bigcup_{j \in I\setminus \{i\}}E_j.
$$
Note that $S_i \subset V_p^* \cap V_q^*$. Given $a\in S_i$, we have that $x_{p,i}(a)=0$ and by Equation \eqref{eq:relacioncoordenadas}, we get $x_{q,i}(a)=x_{q,i_q}(a)=0$, since we are assuming $C_{i_q,i}\ne 0$. This means that $a\in E_{i_q}$, what is a contradiction.
\end{proof}

\begin{defin} \label{defin:conexion}
Let $(\mathcal{M},\mathfrak{a})$ be a monomial generalized analytic manifold, and fix two corner points $p,q \in \mathcal{Z}^0$ such that $I_p \cap I_q \ne \emptyset$. The \emph{weight connexion function from $p$ to $q$} is the map $\gamma^{pq}:I_p \cap I_q \to \mathbb{R}_{>0}$, defined by $i \mapsto C^{pq}_{ii}$, where $C^{pq}\in \mathfrak{C}_{(\mathcal{M},\mathfrak{a})}$. 
\end{defin}
In the previous definition, note that $i \in I_p \cap I_q$ if and only if $p,q \in E_i$. We use the notation $\gamma^{pq}_i=\gamma^{pq}(i)$. In view of Equation \eqref{eq:bastanaristas} and Lemma \ref{lema:triangularinferior}, if $J=I_p \cap I_q$ and $\mathcal{P}=(Y_1,Y_2,\ldots, Y_k)$ is a path of edges from $p$ to $q$ inside $E_J$, we have the relations
\begin{equation} \label{eq:conexioncaminos}
	\gamma_{|_J}^{pq}=\gamma_{|_J}^{p_{k-1}q}\boldsymbol{\cdot}\;\cdots\;\boldsymbol{\cdot} \gamma_{|_J}^{p_1p_2}\boldsymbol{\cdot}\gamma_{|_J}^{pp_1},
\end{equation}
where $p_i=Y_i \cap Y_{i+1}$, for each $i=1,2,\ldots,k-1$.

\begin{lema} \label{lema:conexioninversa}
Given a boundary component $E_i \in \mathcal{Z}^{n-1}$ and two corner points $p,q \in E_i$, we have that $\gamma^{pq}_i\gamma^{qp}_i=1$. 
\end{lema}
\begin{proof}
In view of Equation \eqref{eq:conexioncaminos} it is enough to prove the result for two corner points $p,q$ connected through a compact edge $Y \subset E_i$. Now we have that $(Y,Y)$ is a path from $p$ to $p$ inside $E_i$, and in view of Equation \eqref{eq:conexioncaminos} again, we get
$$
1=\gamma_i^{pp}=\gamma_i^{pq}\gamma_i^{qp},
$$
as we wanted.
\end{proof}

We end the section by introducing the combinatorial data associated to a morphism. 

Let $\phi:(\mathcal{M}_1,\mathfrak{a}_1)\to (\mathcal{M}_2,\mathfrak{a}_2)$ be a morphism of monomial $\mathcal{A}$-manifolds. Given a corner point $p\in \mathcal{Z}^0_{\mathcal{M}_1}$ and $\bar{p}=\phi(p)$, we represent $\phi$ locally at $p$ by means of the \emph{matrix of exponents} $B^\phi_p:I_{\bar{p}}\times I_{p} \to \mathbb{R}_+$, defined by $(i,j) \mapsto b_{ij}:=b_i(j)$, where $b_i:I_p \to \mathbb{R}_+$ is as in Equation \eqref{eq:morfismos}. 
\begin{defin}
	The \emph{combinatorial data $\mathfrak{B}_\phi$ of a morphism $\phi:(\mathcal{M}_1,\mathfrak{a}_1)\to (\mathcal{M}_2,\mathfrak{a}_2)$} is the family of matrices of exponents $\{B^\phi_p\}_{p \in \mathcal{Z}^0_{\mathcal{M}_1}}$.
\end{defin}
\begin{rem}
	Once we fix an affine coordinate system $(V^\star_p,\mathbf{x}_p)$ at some corner point $p\in \mathcal{Z}^0$, we can recover the whole monomial atlas $\mathfrak{a}$ of $\mathcal{M}$ from $\mathfrak{C}_{(\mathcal{M},\mathfrak{a})}$ using Equation \eqref{eq:relacioncoordenadas}. Moreover, a morphism $\phi$ between monomial $\mathcal{A}$-manifolds is completely determined by its combinatorial data $\mathfrak{B}_\phi$.
\end{rem}

\subsection{Abundance of Standardizations of Monomial Manifolds.}

In this section, we give the definition of m-standardization, we give a characterization for its combinatorial data and we prove a result of abundance of m-standardizations of a fixed monomial $\mathcal{G}$-manifold. 

\strut

Throughout this section we fix a monomial generalized analytic manifold $(\mathcal{M},\mathfrak{a})$.

\strut

A \emph{local m-standardization} of $(\mathcal{M},\mathfrak{a})$ at a corner point $p$ is just an affine coordinate system $\mathbf{u}_p$ over $V^\star_p$ such that $\mathbf{u}_p \circ \mathbf{x}_p^{-1}$ is given by the relations
\begin{equation} \label{eq:alpha}
u_{p,i}=x_{p,i}^{\alpha_{p,i}}, \quad \alpha_{p,i} \in \mathbb{R}_{>0}, \quad i\in I_p, 
\end{equation}
where $\mathbf{x}_p \in \mathfrak{a}$. We represent this change of coordinates by means of the map $\alpha_p:I_p \to \mathbb{R}_{>0}$ defined by $i \mapsto \alpha_{p,i}$. In that way, the change of coordinates $\mathbf{u}_{p}\circ \mathbf{x}^{-1}_{p}$ is codified with the \emph{matrix of exponents} $D_{\alpha_p}:I_p \times I_p \to \mathbb{R}_{>0}$. 
\begin{defin}
An \emph{m-standardization of a monomial generalized analytic manifold $(\mathcal{M},\mathfrak{a})$} is a pair $(\mathcal{O},\mathfrak{b})$, where $\mathcal{O}$ is a standardization of $\mathcal{M}$ and $\mathfrak{b}=\{\mathbf{u}_p\}_{p\in \mathcal{Z}^0}$ is a monomial atlas of $\mathcal{N}=(M,\mathcal{O})$ such that $\mathbf{u}_p$ is local m-standardization of $(\mathcal{M},\mathfrak{a})$ for every corner point $p\in \mathcal{Z}^0$. 
\end{defin}

\begin{rem}
If $(\mathcal{O},\mathfrak{b})$ and $(\mathcal{O},\mathfrak{b}')$ are both m-standardizations of the same monomial generalized analytic manifold $(\mathcal{M},\mathfrak{a})$, then we have necessarily that $\mathfrak{b}=\mathfrak{b}'$ as we have already noted in Remark \ref{rem:variosatlasmonomiales}.
\end{rem}

\begin{defin}The \emph{combinatorial data of an m-standardization  ${(\mathcal{O},\mathfrak{b})}$} is given by the collection of maps $\Lambda_{(\mathcal{O},\mathfrak{b})}=\{\alpha_p\}_{p \in \mathcal{Z}^0}$, where, for each corner point $p\in \mathcal{Z}^0$, the map $\alpha_p$ represents, as in Equation \eqref{eq:alpha}, the change of coordinates  $\mathbf{u}_p \circ \mathbf{x}_p^{-1}$, with  $\mathbf{x}_p \in \mathfrak{a}$, and $\mathbf{u}_p \in \mathfrak{b}$. 
\end{defin}

\begin{lema} \label{lema:caractestandmonomial}
A collection of maps $\Lambda=\{\alpha_p: I_p \to \mathbb{R}_{>0}\}_{p \in \mathcal{Z}^0}$ is the combinatorial data of an m-standardization of $(\mathcal{M},\mathfrak{a})$  if and only if for any pair of corner points $p,q \in \mathcal{Z}^0$ the following relations hold:
	\begin{equation}\label{eq:realizable}
		\alpha_{p,\ell}=\gamma^{pq}_\ell\alpha_{q,\ell}, \; \text{ for all }\, \ell \in I_p\cap I_q,
	\end{equation}
	where $\gamma^{pq}$ is the weight connexion function from $p$ to $q$.
\end{lema}
\begin{proof}
Let us assume first that $\Lambda=\Lambda_{(\mathcal{O},\mathfrak{b})}$, where $(\mathcal{O},\mathfrak{b})$ is an m-standardization of $(\mathcal{M},\mathfrak{a})$. Let us denote $\mathcal{N}=(M,\mathcal{O})$ and let $\mathfrak{C}_{(\mathcal{N},\mathfrak{b})}$ be the combinatorial data of the monomial standard analytic manifold $(\mathcal{N},\mathfrak{b})$. In view of Equation \eqref{eq:conexioncaminos}, it is enough to prove Equation \eqref{eq:realizable} for two corner points $p$ and $q$ connected through a compact edge $Y$. Let us consider the affine charts $\mathbf{u}_p,\mathbf{u}_q\in \mathfrak{b}$ at $p$ and $q$ respectively. The change of coordinates $\mathbf{u}_p\circ \mathbf{u}_q^{-1}$ is codified by means of the matrix of exponents $A=A_Y^{pq} \in\mathfrak{C}_{(\mathcal{N},\mathfrak{b})}$. This change must be standard analytic in its domain of definition $\mathbf{u}_q(V^\star_p \cap V^\star_q)=\mathbb{R} \times \mathbb{R}^{n-1}_+$, and this implies
\begin{equation} \label{eq:enteros}
	A_{i \ell}\in \mathbb{Z}_{+}, \quad A^{-1}_{j \ell}\in \mathbb{Z}_{+}, \quad i \in I_{q}, \, j \in I_{p}, \, \ell\in I_Y.
\end{equation}
Let $C=C_Y^{pq} \in \mathfrak{C}_{(\mathcal{M},\mathfrak{a})}$ and $\alpha_p,\alpha_q \in \Lambda_{(\mathcal{O},\mathfrak{b})}$. Note that $A$ is obtained as
$$
A=D_{\alpha_q}CD_{\alpha_p}^{-1}:I_q \times I_p \to \mathbb{R}.
$$
When $\ell \in I_p\cap I_q=I_Y$, in view of Lemma \ref{lema:triangularinferior} and Lemma \ref{lema:conexioninversa}, we have that
$$
A_{\ell\ell}=\frac{\gamma^{pq}_\ell\alpha_{q,\ell}}{\alpha_{p,\ell}}\in \mathbb{Z}_{+}, \quad (A^{-1})_{\ell\ell}=(A^{qp}_Y)_{\ell\ell}=\frac{\gamma^{qp}_\ell\alpha_{p,\ell}}{\alpha_{q,\ell}}=\frac{\alpha_{p,\ell}}{\gamma^{pq}_\ell\alpha_{q,\ell}}=1/A_{\ell\ell}\in \mathbb{Z}_{+},
$$
which shows $A_{\ell\ell}=(A^{-1})_{\ell\ell}=1$. From here we get $\alpha_{p,\ell}=\gamma^{pq}_\ell\alpha_{q,\ell}$, and hence $\Lambda$ satisfies Equation \eqref{eq:realizable} as we wanted.
	
Assume now that $\Lambda$ satisfies Equation \eqref{eq:realizable} for any pair of corner points $p,q \in \mathcal{Z}^0$. For each corner point $p\in \mathcal{Z}^0$, consider affine coordinates $\mathbf{u}_p$ defined on $V^\star_p$ such that the change of coordinates $\mathbf{u}_p\circ\mathbf{x}^{-1}_p$ satisfies $u_{p,j}=x_{p,j}^{\alpha_{p,j}}$, for all $j \in I_p$, where $\alpha_p \in \Lambda$ and $\mathbf{x}_p \in \mathfrak{a}$. In that way, we get a new monomial atlas $\mathfrak{b}=\{\mathbf{u}_p\}_{p\in \mathcal{Z}^0}$ of $\mathcal{M}$. Let us see that the changes of coordinates $\mathbf{u}_q \circ \mathbf{u}_p^{-1}$ are standard analytic for any pair of corner points $p$ and $q$. In view of Equation \eqref{eq:bastanaristas} it is enough to suppose that $p$ and $q$ are connected through an edge $Y$. Defining the matrix $A=D_{\alpha_q}CD^{-1}_{\alpha_p}$, the change of coordinates $\mathbf{u}_q \circ \mathbf{u}_p^{-1}$ is given by
$$
u_{q,i}=\prod_{j\in I_p}u_{p,j}^{A_{ij}}, \; \mbox{ for any } i\in I_q.
$$
It suffices to show that $A$ satisfies the conditions in Equation \eqref{eq:enteros}. Indeed, if $A_{i\ell} \in \mathbb{Z}_+$ for $i\in I_q$ and for all $\ell\in I_Y$, then $u_{q,i}$ in the above equation is standard analytic in the domain $V^\star_q\cap \{u_{p,i_p}\ne 0\}=V^\star_q\cap V^\star_p$.

Applying Lemma \ref{lema:triangularinferior} we get that $A_{r\ell}=0$ and also that $(A^{-1})_{r \ell }=0$, for all $r,\ell \in I_Y$ with $r \ne \ell$. Moreover, the mentioned result assures that $C_{i_q\ell}=0$ and that $(C_{i_p\ell})^{-1}=C^{qp}_{i_p\ell}=0$, for all $\ell \in I_Y$; hence, for any such an index $\ell\in I_Y$ we obtain
$$
A_{i_q\ell}=C_{i_q\ell}\alpha_{q,i_q}/\alpha_{p,\ell}=0, \quad (A^{-1})_{i_p\ell}=(C_{i_p\ell})^{-1}\alpha_{p,i_p}/\alpha_{q,\ell}=0.
$$
Also in view of Lemma \ref{lema:triangularinferior} we get
$$
A_{\ell\ell}=\frac{C_{\ell\ell}\alpha_{q,\ell}}{\alpha_{p,\ell}}=\frac{\gamma_{\ell}^{pq}\alpha_{q,\ell}}{\alpha_{p,\ell}}, \quad (A^{-1})_{\ell\ell}=\frac{(C^{-1})_{\ell\ell}\alpha_{p,\ell}}{\alpha_{q,\ell}}=\frac{\gamma_{\ell}^{qp}\alpha_{p,\ell}}{\alpha_{q,\ell}},
$$
for all $\ell \in I_Y$, and applying Lemma \ref{lema:conexioninversa} and the conditions in Equation \eqref{eq:realizable} we conclude $A_{\ell\ell}=(A^{-1})_{\ell\ell}=1$. Once that we know that all the changes of coordinates $\mathbf{u}_q\circ \mathbf{u}_p^{-1}$ are standard analytic, we have that $\mathfrak{b}$ is an atlas defining a standard analytic structure $\mathcal{N}=(M,\mathcal{O})$ over $M$, where $\mathcal{M}=(M,\mathcal{G}_M)$; thus $\mathcal{O}\subset \mathcal{G}_M$ is a standardization of $\mathcal{M}$. Moreover, by definition of $\mathfrak{b}$, we have that $(\mathcal{O},\mathfrak{b})$ is an m-standardization of $(\mathcal{M},\mathfrak{a})$ with $\Lambda_{(\mathcal{O},\mathfrak{b})}=\Lambda$.
\end{proof}
In the sequel, a collection of maps $\Lambda=\{\alpha_p:I_p \to \mathbb{R}_{> 0}\}_{p \in \mathcal{Z}^0}$  is called \emph{realizable for $(\mathcal{M},\mathfrak{a})$} if Equation \eqref{eq:realizable} holds for any pair of corner points $p,q \in \mathcal{Z}^0$.

\begin{defin}
Let $\mathbf{u}_p$ be a local m-standardization of $(\mathcal{M},\mathfrak{a})$ at a given corner point $p\in \mathcal{Z}^0$. An \emph{extension of $\mathbf{u}_p$} is a (global) m-standardization $(\mathcal{O},\mathfrak{b})$ of $(\mathcal{M},\mathfrak{a})$ such that $\mathbf{u}_p \in \mathfrak{b}$, we say also that \emph{$(\mathcal{O},\mathfrak{b})$ extends $\mathbf{u}_p$}. We denote by $\mathcal{E}(\mathbf{u}_p)$ the set of m-standardizations extending $\mathbf{u}_p$.
\end{defin}

\begin{prop} \label{prop:extension}
Let $(\mathcal{M},\mathfrak{a})$ be a monomial generalized analytic manifold. Then:
\begin{enumerate}
	\item[a)] There is a bijection between the set of m-standardizations of $(\mathcal{M},\mathfrak{a})$  and $\mathbb{R}^N_{>0}$, where $N$ is the number of boundary components of $\partial M$.
	\item[b)] Given a corner point $p\in \mathcal{Z}^0$ and a local m-standardization $\mathbf{u}_p$ at $p$, there is a bijective map $\mathbb{R}^{N-n}_{>0} \to \mathcal{E}(\mathbf{u}_p)$, where $n$ is the dimension of $\mathcal{M}$.
\end{enumerate}
\end{prop}
\begin{proof}
We start with the proof of the first assertion. Let $I$ be the set of indices labelling the components of $\partial M$, that is $\partial M=\cup_{i\in I} E_i$, where $N=\#I$, and let us fix a collection of corner points $\mathfrak{q}=\{q_i\}_{i\in I}$ in such a way that $q_i \in E_i$ for each $i \in I$. Given a map $\beta \in \mathbb{R}^I_{>0}$, we take $\Lambda_\beta=\{\alpha_p\}_{p\in \mathcal{Z}^0}$ to be the family of maps defined by
$$
\alpha_{p,\ell}=\gamma^{pq_\ell}_\ell\beta_\ell, \; \text{ for all } p\in \mathcal{Z}^0 \, \text{ and }\, \ell\in I_p. 
$$
Let us see that $\Lambda_\beta$ is a realizable family of maps; that is, given two corner points $p$ and $q$ in $\mathcal{Z}^0$ we want to prove that Equation \eqref{eq:realizable} holds. Let us fix such a corner points $p$ and $q$, and let $\ell\in I_p \cap I_q$. By Equation \eqref{eq:conexioncaminos} we have that $\gamma^{pq_\ell}_\ell\gamma^{q_\ell q}_\ell=\gamma^{pq}_\ell$. Moreover, by the definition of $\alpha_{q,\ell}$ and as a consequence of Lemma \ref{lema:conexioninversa} we have that $\beta_\ell=\gamma^{q_\ell q}_\ell\alpha_{q,\ell}$. Then we obtain
$$
\alpha_{p,\ell}=\gamma^{pq_\ell}_\ell\beta_\ell=\gamma^{pq_\ell}_\ell\gamma^{q_\ell q}_\ell\alpha_{q,\ell}=\gamma_\ell^{pq}\alpha_{q,\ell}
$$
as we wanted. Now, in view of Lemma \ref{lema:caractestandmonomial}, there exists a unique m-standardization $(\mathcal{O}^\beta,\mathfrak{b}^\beta)$ with $\Lambda_{(\mathcal{O}^\beta,\mathfrak{b}^\beta)}=\Lambda_{\beta}$. Finally, note that the map 
$$
\Psi_{\mathfrak{q}}: \mathbb{R}^I_{>0} \to \left\{\begin{array}{c}
	\text{m-standardizations} \\ \text{ of } (\mathcal{M},\mathfrak{a})
\end{array}
\right\}, \quad \beta \mapsto (\mathcal{O}^\beta,\mathfrak{b}^\beta)
$$
is a bijection. Indeed, if $\beta \ne \beta'$, we have that $\Lambda_\beta \ne \Lambda_{\beta'}$ and hence $(\mathcal{O}^\beta,\mathfrak{b}^\beta) \ne (\mathcal{O}^{\beta'},\mathfrak{b}^{\beta'})$. On the other hand, given an m-standardization $(\mathcal{O},\mathfrak{b})$ with combinatorial data $\Lambda=\{\alpha_p\}_{p\in \mathcal{Z}^0}$, we have that $(\mathcal{O},\mathfrak{b})=\Psi_{\mathfrak{q}}(\beta)$, where $\beta$ is defined by $\beta_i=\alpha_{q_i,i}$, for all $i\in I$. The proof of a) is finished.

Let us prove now the second assertion b). Denote by $\alpha_p:I_p \to \mathbb{R}_{>0}$ the map of exponents defining $\mathbf{u}_p$, that is, for all $i\in I_p$ we have
$$
u_{p,i}=x_{p,i}^{\alpha_{p,i}},\; \text{ where }\mathbf{x}_p\in \mathfrak{a}.
$$
Consider the inclusion map $i_{\alpha_p}:\mathbb{R}^{I\setminus I_p}_{>0} \hookrightarrow \mathbb{R}^{I}_{>0}$, defined by 
$$
\delta \mapsto i_{\alpha_p}(\delta):=\beta^\delta, \; \text{ where } \beta^\delta_i=\left\{
\begin{array}{lcl}
\delta_i & \text{ if } & i\in I\setminus I_p, \\
\alpha_{p,i}& \text{ if } & i\in I_p.
\end{array}
\right.
$$
Take $\mathfrak{q}_p=\{q_i\}_{i\in I}$ such that $q_i=p$, for each $i \in I_p$, and $q_i \in E_i$, for all $i\in I\setminus I_p$. By definition, we note that $\Psi_{\mathfrak{q}_p}(\beta) \in \mathcal{E}(\mathbf{u}_p)$ if and only if $\beta|_{I_p}=\alpha_{p}$, or equivalently $\beta =i_{\alpha_p} (\beta |_{I\setminus I_p})$. In other words, we have the equality
$$
\mathcal{E}(\mathbf{u}_p)=\text{Im}(\Psi_{\mathfrak{q}_p} \circ i_{\alpha_p}),
$$
and hence we have the bijection $\mathbb{R}_{>0}^{I\setminus I_p} \to \mathcal{E}(\mathbf{u}_p)$ mapping $\delta$ into $\Psi_{\mathfrak{q}_p}(i_{\alpha_p}(\delta))$. We finish just by noting that $\#I_p=n$.
\end{proof}

\subsection{The Monomial Voûte Etoilée}
We devote this section to give the definition of m-combinatorial blowing-up and to introduce the concept of ``monomial voûte étoilée'' over an m-manifold , whose elements, called m-stars, are sequences of monomial blowing-ups starting from that m-manifold. The terminology is inspired by Hironaka \cite{Hir}. 

\strut

Let $(\mathcal{M},\mathfrak{a})$ be a monomial generalized analytic manifold. An \emph{m-combinatorial center of blowing-up for $(\mathcal{M},\mathfrak{a})$} is a tripet $(Z,\mathcal{O},\mathfrak{b})$, where $Z$ is a combinatorial geometrical center for $\mathcal{M}$ and $(\mathcal{O},\mathfrak{b})$ is an m-standardization of $(\mathcal{M},\mathfrak{a})$. Given such an m-combinatorial center $(Z, \mathcal{O},\mathfrak{b})$, we consider the blowing-up $\pi_\xi:\mathcal{M}_\xi \to \mathcal{M}$ with center $\xi=(Z,\mathcal{O})$. Let $I$ be an index set labelling the components of $\partial M$. We write $\infty \notin I$ to label the exceptional divisor $E_\infty=\pi_\xi^{-1}(Z)$, thus $I_\xi=I \cup \{\infty\}$ is an index set labelling the components of $\partial M_\xi$. More precisely, given $i\in I$, it represents both the boundary components 
$$
E_i \subset \partial M \;\text{ and }\; E'_i=\overline{\pi^{-1}_\xi(E_i\setminus Z)} \subset \partial M_\xi,
$$
belonging to $\mathcal{Z}_{\mathcal{M}}^{n-1}$ and $\mathcal{Z}_{\mathcal{M}_\xi}^{n-1}$, respectively. The index $\infty\in I_\xi$ represents $E_\infty\in \mathcal{Z}_{\mathcal{M}_\xi}^{n-1}$.

\begin{prop} \label{prop:explosionmonomial}
 There is a monomial atlas $\mathfrak{a}_\xi$ for $\mathcal{M}_\xi$, in such a way that $\pi_\xi$ defines a morphism of monomial $\mathcal{G}$-manifolds from $(\mathcal{M}_\xi,\mathfrak{a}_\xi)$ to $(\mathcal{M},\mathfrak{a})$. 
\end{prop}
\begin{proof}
Take a corner point $p'$ in $\mathcal{M}_\xi$ and let $p=\pi_\xi(p')$. Note that $p$ is a corner point in $\mathcal{M}$. Let $\mathbf{x}_p \in \mathfrak{a}$ be the affine chart of the atlas $\mathfrak{a}$ at $p$. We distinguish two situations:

\emph{Case $p'\notin E_\infty$}. We have that $I_{p'}=I_p$ and the blowing-up $\pi_\xi$ induces an isomorphism between $V_p^*$ and $V_{p'}^*$. We take affine coordinates $\mathbf{x}'_{p'}$ over $V^\star_{p'}$ defined by
$$
\tilde{x}'_{p',i}=x_{p,i} \circ \pi_\xi|_{V_{p'}^*}, \; \text{for all } i\in I_p
$$
Thus, the expression of $\pi_\xi$ in coordinates $\mathbf{x}'_{p'}$ and $\mathbf{x}_p$ is purely monomial.
This expression can be codified with the matrix of exponents $B_{p'}:I_{p}\times I_{p} \to \mathbb{R}_{\geq 0}$ given by
\begin{equation}\label{eq:identidad}
	B_{p'}(i,j)=\delta_{ij}, \quad i,j \in I_p,
\end{equation}
where $\delta_{ij}$ is the Kronecker delta symbol. In other words, $B_{p'}=D_{\mathds{1}_{I_p}}$.

\emph{Case $p'\in E_\infty$}. We have that $I_{p'}=I_p\setminus \{j\}\cup \{\infty\}$, for some $j \in I_Z$ (see for instance \cite{Mol} for details in the combinatorial treatment of blowing-ups). By hypothesis the pair $(\mathcal{O},\mathfrak{b})$ is an m-standardization of $(\mathcal{M},\mathfrak{a})$, in particular $\mathfrak{b}$ is a monomial atlas of the standard analytic manifold $\mathcal{N}=(M,\mathcal{O})$. Using this information together with the definition of blowing-up centered at $\xi$, we get that there exists affine coordinate systems $\mathbf{x}'_{p'}$ defined in $V^\star_{p'}$ such that the map $\mathbf{x}'_{p'} \circ \pi_\xi \circ \mathbf{x}^{-1}_p$ is purely monomial with associated matrix of exponents $B_{p'}:{I_p \times I_{p'}} \to \mathbb{R}_+$ given by
\begin{equation} \label{eq:blowup}
(r,s) \mapsto \left\{
\begin{array}{lclclccc}
	1 &\text{ if } & r=s & \text{ and } & r\in I_p \setminus\{j\}, \\
	\alpha_{p,j}/\alpha_{p,r} & \text{ if } & s=\infty& \text{ and }& r \in I_Z, &  \\
	0 & & \text{otherwise}.
\end{array}	
\right.
\end{equation}
where $\alpha_p \in \Lambda_{(\mathcal{O},\mathfrak{b})}$. With an appropriated order of rows and columns, $B_{p'}$ can be seen as the upper triangular matrix 
$$
\left(
\begin{array}{c | c | c }
	\text{Id}_{n-s}   & 0 &  0 \\
	\hline 
	0 & \text{Id}_{s-1} & \mathbf{a} \\
	\hline 
	0    &0    & 1 
\end{array}
\right)
\in \mathbb{R}_+^{n \times n},
$$
where $s=\# I_Z$ and $\mathbf{a} \in \mathbb{R}_{>0}^{s-1}$ is a column vector whose entries are defined by the quotients $\alpha_{p,j}/\alpha_{p,r}$, with $r \in I_Z$.

The collection $\mathfrak{a}_\xi=\{\mathbf{x}'_{p'}\}_{p'\in \mathcal{Z}^0_\xi}$ with $\mathcal{Z}^0_\xi=\mathcal{Z}^0_{\mathcal{M}_\xi}$, is thus a monomial atlas in $\mathcal{M}_\xi$. Moreover the blowing-up $\pi_\xi$ induces a morphism from $(\mathcal{M}_\xi,\mathfrak{a}_\xi)$ to $(\mathcal{M},\mathfrak{a})$ and the associated combinatorial data is $\mathfrak{B}_{\pi_\xi}=\{B_{p'}\}_{p'\in \mathcal{Z}^0_\xi}$.
\end{proof}
From now on, given an m-combinatorial center of blowing-up $(Z, \mathcal{O},\mathfrak{b})$ for a monomial generalized analytic manifold $(\mathcal{M},\mathfrak{a})$, and the blowing-up morphism $\pi_\xi:\mathcal{M}_\xi \to \mathcal{M}$, with center at $\xi=(Z,\mathcal{O})$, we always consider $\mathcal{M}_\xi$ endowed with the monomial atlas $\mathfrak{a}_\xi$ constructed in Proposition \ref{prop:explosionmonomial}. Moreover, we also write
$$
\pi_\xi:(\mathcal{M}_\xi,\mathfrak{a}_\xi) \to (\mathcal{M},\mathfrak{a}),
$$
to emphasize that the morphism $\pi_\xi$ is considered also as a morphism of monomial generalized analytic manifolds, and we call it an \emph{m-combinatorial blowing-up of $(\mathcal{M},\mathfrak{a})$}. The associated combinatorial data $\mathcal{B}_{\pi_\xi}=\{B_{p'}\}_{p'\in \mathcal{Z}_\xi}$ of this morphism has been made explicit with Equation \eqref{eq:identidad}, for points $p'\in \mathcal{Z}_\xi^0$ with $p'\not\in E_\infty$ and with Equation \eqref{eq:blowup}, for points $p'\in \mathcal{Z}_\xi^0(E_\infty)$. 

\begin{defin} Let $(\mathcal{M},\mathfrak{a})$ be a monomial generalized analytic manifold. An \emph{m-star over $(\mathcal{M},\mathfrak{a})$} is the composition $\sigma=\pi_0\circ \pi_1\circ \cdots \circ \pi_{r-1}$ of a finite sequence of m-combinatorial blowing-ups. That is
$$
	\sigma: (\mathcal{M}_{r},\mathfrak{a}_{r}) \xrightarrow{\pi_{r-1}} (\mathcal{M}_{r-1}, \mathfrak{a}_{r-1}) \xrightarrow{\pi_{r-2}} \cdots \xrightarrow{\pi_{0}} (\mathcal{M}_0,\mathfrak{a}_0)=(\mathcal{M},\mathfrak{a}), 
$$
where for each $k=0,1,2,\ldots,r-1$, the morphism $\pi_k:=\pi_{\xi_k}$ is the blowing-up with respect to $\xi_k=(Z_k,\mathcal{O}_k,\mathfrak{b}_k)$, an m-combinatorial center of blowing-up for $(\mathcal{M}_{k}, \mathfrak{a}_{k})$. The integer $r$ is called the \emph{age} of the m-star $\sigma$ and the monomial generalized analytic manifold $(\mathcal{M}_{r},\mathfrak{a}_{r})$ is the \emph{end} of the m-star. The collection $\mathcal{V}^{\text{m}}_{(\mathcal{M},\mathfrak{a})}$ of all the m-stars over $(\mathcal{M},\mathfrak{a})$ is the \emph{monomial voûte étoilée of $(\mathcal{M},\mathfrak{a})$}.
\end{defin}

\section{Stratified Reduction of Singularities via Principalization of m-ideals}

In this section we introduce the concept of m-ideal, in order to prove a theorem of principalization. With this result we prove the stratified reduction of singularities for a global function defined in generalized analytic manifolds admitting a monomial structure. We apply this result to prove the main result of this paper stated in Theorem \ref{th:main}.

\subsection{Principalization of m-Ideals}

Let us fix a monomial generalized analytic manifold $(\mathcal{M},\mathfrak{a})$, where $\mathcal{M}=(M,\mathcal{G}_M)$. 

Take a global function $f \in \mathcal{G}_M(M)$ and two corner points $p,q\in \mathcal{Z}^0$. Let $\mathbf{x}_q, \mathbf{x}_q \in \mathfrak{a}$ be the affine coordinates at $p$ and $q$, respectively, and let $C^{pq}\in \mathfrak{C}_{(\mathcal{M},\mathfrak{a})}$ be the matrix of exponents codifying the change of coordinates $\mathbf{x}_q\circ \mathbf{x}^{-1}_p$. We have the relation
\begin{equation} \label{eq:relacionsoporte}
\text{Supp}_{p}(f;\mathbf{x}_{p})=\{\lambda_qC^{pq}: \; \lambda_q \in \text{Supp}_{q}(f;\mathbf{x}_{q})\} \subset \mathbb{R}^{I_p}_+,
\end{equation}
where the product has been done passing through the natural identification of $\mathbb{R}^{I_q}$ with $\mathbb{R}^{\{\cdot\} \times I_q}$.
\begin{defin}
A generalized analytic global function $\mathfrak{m} \in \mathcal{G}_M(M)$ is said to be an \emph{m-function in $(\mathcal{M},\mathfrak{a})$} if for each $p\in \mathcal{Z}^0$, there is a map $\lambda_p:I_p \to \mathbb{R}_{+}$, such that 
$$
\mathfrak{m}|_{V^\star_p}=\mathbf{x}_p^{\lambda_p}, \; \text{ where }\mathbf{x}_p\in \mathfrak{a}
$$
The \emph{combinatorial data of $\mathfrak{m}$} is the list $\mathcal{L}_\mathfrak{m}=\{\lambda_p\}_{p\in \mathcal{Z}^0}$.
\end{defin}

Let us consider an m-function $\mathfrak{m}$ in $(\mathcal{M},\mathfrak{a})$ with combinatorial data $\mathcal{L}_{\mathfrak{m}}=\{\lambda_p\}_{p\in \mathcal{Z}^0}$. By Equation \eqref{eq:relacionsoporte},  for any pair of corner points $p,q \in \mathcal{Z}^0$ we have the relation $\lambda_p=\lambda_qC^{pq}$,  where $C^{pq}\in \mathfrak{C}_{(\mathcal{M},\mathfrak{a})}$. In particular, if we consider the set of common indices $J=I_p\cap I_q$, we get that  
\begin{equation} \label{eq:cambiom-funciones}
	\lambda_{p,\ell}=\lambda_{q,\ell}\gamma^{pq}_\ell, \; \text{ for any } \ell\in J,
\end{equation}
where$\gamma^{pq}$ is the weighted connexion function from $p$ to $q$. Indeed, it is enough to check above relation for the case where $p$ and $q$ are connected through an edge $Y$, in view of Equation \eqref{eq:conexioncaminos}. For this case, it holds as a consequence of Lemma \ref{lema:triangularinferior}.

\begin{rem} \label{rem:list}
Given a list of maps $\mathcal{L}=\{\lambda_p: I_p\to \mathbb{R}_+\}_{p\in \mathcal{Z}^0}$ satisfying $\lambda_p=\lambda_qC^{pq}$, for any pair of corner points $p,q \in \mathcal{Z}^0$, there exists an m-function $\mathfrak{m}$ such that $\mathcal{L}_\mathfrak{m}=\mathcal{L}$. 
\end{rem}
%

\begin{defin}
A \emph{finitely generated m-ideal $\mathcal{J}$ in $(\mathcal{M},\mathfrak{a})$} is a sheaf of ideals $\mathcal{J}\subset \mathcal{G}_M$ given by
$$
\mathcal{J}=\mathfrak{m}_1\mathcal{G}_M+\mathfrak{m}_2\mathcal{G}_M+\cdots+\mathfrak{m}_k\mathcal{G}_M=:(\mathfrak{m}_1,\mathfrak{m}_2,\ldots,\mathfrak{m}_k)
$$
where $\mathfrak{m}_1,\mathfrak{m}_2,\ldots,\mathfrak{m}_k$ are m-functions called \emph{m-generators of $\mathcal{J}$}.
\end{defin}
\begin{rem}
	In general we haven't got noetherianity for the rings of generalized analytic functions, then it is relevant to ask the finitely generated condition to $\mathcal{J}$.
\end{rem}

\begin{notation}
	Let $I$ be a finite index set and let $A$ be a finite subset of $\mathbb{R}^I$. We write $A^{\min}$ to denote the set of elements in $A$ that are minimal with respect to the division order $\leq_d$ in $\mathbb{R}^I$.
\end{notation}

Let $\mathcal{J}$ be an m-ideal in $(\mathcal{M},\mathfrak{a})$ with set of m-generators $G=\{\mathfrak{m}_1,\mathfrak{m}_2,\ldots,\mathfrak{m}_k\}$. For each $i=1,2,\ldots,k$, let us write  $\mathcal{L}_{\mathfrak{m}_i}=\{\lambda_p^i\}_{p\in \mathcal{Z}^0}$. Given a corner point $p\in \mathcal{Z}^0$ and $\mathbf{x}_p\in \mathfrak{a}$ the affine coordinates in $p$, the restriction $\mathcal{J}|_{V^\star_p}$ is an m-ideal in the m-corner $(\mathcal{M}|_{V^\star_p},\mathbf{x}_p)$ generated by the restriction $G|_{V^\star_p}=\{\mathfrak{m}_1|_{V^\star_p},\mathfrak{m}_2|_{V^\star_p},\ldots,\mathfrak{m}_k|_{V^\star_p}\}$. The \emph{combinatorial data of $G|_{V^\star_p}$} is the set 
$$
\Gamma_{G,p}:=\{\lambda^1_{p},\lambda^2_{p},\ldots,\lambda^k_{p}\} \subset \mathbb{R}_+^{I_p}.
$$
Note that if $(\Gamma_{G,p})^{\min}=\{\mu_p^{1},\mu_p^{2},\ldots,\mu_p^{{k_p}}\}$, then
\begin{equation} \label{eq:restriccionestrella}
\mathcal{J}|_{V^\star_p}=\mathbf{x}_p^{\mu_p^{1}}\mathcal{G}_M|_{V^\star_p}+\mathbf{x}_p^{\mu_p^{2}}\mathcal{G}_M|_{V^\star_p}+\cdots+\mathbf{x}_p^{\mu_p^{{k_p}}}\mathcal{G}_M|_{V^\star_p}
\end{equation}
The sheaf of ideals $\mathcal{J}$ is called \emph{locally principal} if at each point $a\in M$ we have that $\mathcal{J}_a \subset \mathcal{G}_{M,a}$ is a principal ideal. Since $\mathcal{J}$ is an m-ideal, it is enough to ask this property for the corner points. In terms of the combinatorial data, we have that $\mathcal{J}$ is locally principal if and only if the set $(\Gamma_{G,p})^{\min}$ is a singleton for any $p\in \mathcal{Z}^0$, where $G$ is a set of m-generators of $\mathcal{J}$.

\strut

Let $\mathfrak{m}$ be an m-function in $(\mathcal{M},\mathfrak{a})$ and take an m-star $\sigma: (\mathcal{M}',\mathfrak{a}')\to (\mathcal{M},\mathfrak{a})$. The \emph{total transform} $\sigma^*\mathfrak{m}=\mathfrak{m}\circ \sigma$ of $\mathfrak{m}$ by $\sigma$ is a again an m-function in $(\mathcal{M}',\mathfrak{a}')$. More precisely, if  $p'\in \mathcal{Z}^0_{\mathcal{M}'}$ and $p=\sigma(p')$, then $\lambda'_{p'}\in \mathcal{L}_{\sigma^*\mathfrak{m}}$ is given by 
\begin{equation}\label{eq:mfuncionexplosion}
	\lambda'_{p'}=\lambda_pB^\sigma_{p'},
\end{equation}
where $\lambda_p\in \mathcal{L}_\mathfrak{m}$ and $B^\sigma_{p'}\in \mathfrak{B}_\sigma$ is the matrix of exponents representing $\sigma$ at $p'$. If $\mathcal{J}$ is an m-ideal generated by $G=\{\mathfrak{m}_1,\mathfrak{m}_2,\ldots,\mathfrak{m}_k\}$, observe that the total transform $\sigma^*\mathcal{J}$ is also an m-ideal in $(\mathcal{M}',\mathfrak{a}')$ generated by $\sigma^*G:=\{\sigma^*\mathfrak{m}_1,\sigma^*\mathfrak{m}_2,\ldots,\sigma^*\mathfrak{m}_k\}$.

\strut

The main result in this section is the following one about principalization of m-ideals. It can be seen as a version in the category of generalized analytic manifolds of a well-known result in analytic or algebraic geometry (see for instance \cite{Gow}).

\begin{teo}\label{teo:principalizacion}
Let $\mathcal{J}$ be a finitely generated m-ideal in a monomial generalized analytic manifold $(\mathcal{M},\mathfrak{a})$. There exists an m-star $\sigma\in \mathcal{V}^m_{(\mathcal{M},\mathfrak{a})}$ such that $\sigma^*\mathcal{J}$ is locally principal.
\end{teo}

To prove this theorem, we can reduce ourselves to the case where $\mathcal{J}$ is generated by two $m$-functions by considering a clear finite recurrence and the following lemma.  

\begin{lema} \label{cor:bastandosmonomios}
Let $\mathcal{J}=(\mathfrak{m}_1,\mathfrak{m}_2,\ldots,\mathfrak{m}_k)$ be an m-ideal in $(\mathcal{M},\mathfrak{a})$. Assume that $\mathcal{J}_{rs}$ is locally principal for any pair of indices $r,s \in \{1,2,\ldots,k\}$, where $\mathcal{J}_{rs}$ is the m-ideal generated by $\mathfrak{m}_r$ and $\mathfrak{m}_s$. Then $\mathcal{J}$ is locally principal.
\end{lema}

\begin{proof}
Assume that there is a point $p\in \mathcal{Z}^0$ such that $\Gamma_{\mathcal{J},p}$ has not a minimum. There exist indices $r,s \in \{1,2,\ldots,k\}$ such that $\lambda^r_{p}\in \mathcal{L}_{\mathfrak{m}_r}$ and $\lambda^s_{p}\in \mathcal{L}_{\mathfrak{m}_s}$ are not comparable for the division order $\leq_d$ in $\mathbb{R}^{I_{p}}$. Note that $\Gamma_{\mathcal{J}_{rs},p}=\{\lambda^r_{p},\lambda^s_{p}\}$, and hence $\mathcal{J}_{rs}$ is not locally principal, which is a contradiction.
\end{proof}

\paragraph{The Case of Two Generators.} Let us assume that $\mathcal{J}$ is an m-ideal in $(\mathcal{M},\mathfrak{a})$ generated by two m-functions $\mathcal{J}=(\mathfrak{m},\mathfrak{n})$. We write $\mathcal{L}_\mathfrak{m}=\{\lambda_p\}_{p\in \mathcal{Z}^0}$ and $\mathcal{L}_\mathfrak{n}=\{\mu_p\}_{p\in \mathcal{Z}^0}$.

\strut

We introduce first several definition, mainly inspired by the ``b-invariant'' introduced in \cite{vdD-Spe} by van den Dries and Speissegger.

\strut

Let $Z\in \mathcal{Z}^{n-2}$ be a codimension two combinatorial geometrical center for $\mathcal{M}$. We know that the index set $I_Z$ has just two elements, that is $I_Z=\{i,j\}$. Let $p\in Z$ be a corner point. We say that $Z$ is \emph{uncoupled for $\mathcal{J}$ at $p$} if 
$$
(\lambda_{p,i}-\mu_{p,i})(\lambda_{p,j}- \mu_{p,j}) < 0,
$$
We say that $Z$ is \emph{uncoupled for $\mathcal{J}$} if it is so at each corner point $p\in \mathcal{Z}^0(Z)$. 

\begin{lema} \label{lema:desacoplado}
A combinatorial geometrical center $Z\in \mathcal{Z}^{n-2}$ is uncoupled for $\mathcal{J}$ if and only if there is a corner point $q \in \mathcal{Z}^0(Z)$ such that $Z$ is uncoupled for $\mathcal{J}$ at $q$.
\end{lema}
\begin{proof}
Assume that $Z$ is uncoupled for $\mathcal{J}$ at a corner point $q\in \mathcal{Z}^0(Z)$ and take any other corner point $p \in \mathcal{Z}^0(Z)$. By Equation \eqref{eq:cambiom-funciones}, we have that $\lambda_{p,\ell}=\lambda_{q,\ell}\gamma^{pq}_\ell$, for all $\ell \in I_p \cap I_q$, where $\gamma^{pq}$ is the weighted connexion function from $p$ to $q$. Being $p,q \in Z$, we know that $I_Z=\{i,j\}\subset I_p\cap I_q$. Then
$$
(\lambda_{p,i}-\mu_{p,i})(\lambda_{p,j}- \mu_{p,j})=\gamma^{pq}_i\gamma^{pq}_j(\lambda_{q,i}-\mu_{q,i})(\lambda_{q,j}- \mu_{q,j}) <0,
$$
since $\gamma^{pq}_i>0$ and $\gamma^{pq}_j > 0$. Hence $Z$ is uncoupled for $\mathcal{J}$ also at $p$, and we conclude that $Z$ is uncoupled for $\mathcal{J}$.
\end{proof}

Observe that if there are not uncoupled centers for $\mathcal{J}$ passing through a given corner point $p\in \mathcal{Z}^0$, then we necessarily have that $\lambda_p\leq_d \mu_p$ or $\mu_p\leq \lambda_p$, that is $\mathcal{J}_p$ is a principal ideal. 
\begin{defin}
We consider $\Omega_\mathcal{J}$ to be the family of codimension two combinatorial geometrical centers in $\mathcal{M}$ that are uncoupled for $\mathcal{J}$, and define \emph{the invariant of $\mathcal{J}$} to be $\text{Inv}_\mathcal{J}:=\#\Omega_\mathcal{J}$.
\end{defin}
We have that $\text{Inv}_\mathcal{J}=0$ if and only if $\mathcal{J}$ is locally principal. Thus, the objective now is to find an m-star $\sigma \in \mathcal{V}^m_{(\mathcal{M},\mathfrak{a})}$ such that $\text{Inv}_{\sigma^*\mathcal{J}}=0$.

\strut

Suppose that $\text{Inv}_\mathcal{J}>0$ and fix $Z \in \Omega_\mathcal{J}$. Take a corner point $p \in Z$ and pick also a local m-standardization $\mathbf{u}_p$ of $(\mathcal{M},\mathfrak{a})$ at $p$ defined by the map $\alpha_p:I_p \to \mathbb{R}_{>0}$. We say that $\mathbf{u}_p$ is \emph{adapted to $\mathcal{J}$ with respect to $Z$} if 
$$
\alpha_{p,j}(\lambda_{p,i}-\mu_{p,i})+\alpha_{p,i}(\lambda_{p,j}-\mu_{p,j})=0
$$
We say that a global m-standardization of $(\mathcal{M},\mathfrak{a})$ is \emph{adapted to $\mathcal{J}$ with respect to $Z$} if it is so at each corner point $p \in Z$. 
\begin{lema} \label{lema:estandadaptada}
An m-standardization $(\mathcal{O},\mathfrak{b})$ is adapted to $\mathcal{J}$ with respect to $Z$ if and only if there is a corner point $q\in Z$ such that $\mathbf{u}_q \in \mathfrak{b}$ is adapted to $\mathcal{J}$ with respect to $Z$.
\end{lema}
\begin{proof}
Denote $\Lambda=\Lambda_{(\mathcal{O},\mathfrak{b})}$. Assume that there is a corner point $q \in Z$ such that $\mathbf{u}_q \in \mathfrak{b}$ is adapted to $\mathcal{J}$ with respect to $Z$. Take any other corner point $p\in Z$. In view of the realizability of $\Lambda$ established in Lemma \ref{lema:caractestandmonomial} we know that $\alpha_{q,\ell}=\gamma^{qp}_\ell\alpha_{p,\ell}$, for all $\ell \in I_p \cap I_q$. Since $p,q \in Z$ we have that $I_Z=\{i,j\} \subset I_p\cap I_q$. Then, by Equation \eqref{eq:cambiom-funciones}, we get
$$
\alpha_{p,j}(\lambda_{p,i}-\mu_{p,i})+\alpha_{p,i}(\lambda_{p,j}-\mu_{p,j})=
\gamma^{pq}_j\alpha_{q,j}\gamma^{pq}_i(\lambda_{q,i}-\mu_{q,i})+\gamma^{pq}_i\alpha_{q,i}\gamma^{pq}_j(\lambda_{q,j}-\mu_{q,j})=
$$
$$
=\gamma^{pq}_i\gamma^{pq}_j[\alpha_{q,j}(\lambda_{q,i}-\mu_{q,i})+\alpha_{q,i}(\lambda_{q,j}-\mu_{q,j})]=0.
$$
As a consequence, the local m-standardization $\mathbf{u}_p \in \mathfrak{b}$ is adapted to $\mathcal{J}$ with respect to $Z$ at $p$. We conclude that $(\mathcal{O},\mathfrak{b})$ is adapted to $\mathcal{J}$ with respect to $Z$.
\end{proof}

A codimension two combinatorial center of blowing-up $\xi=(Z,\mathcal{O},\mathfrak{b})$ is \emph{adapted to $\mathcal{J}$} if $Z \in \Omega_\mathcal{J}$ and $(\mathcal{O},\mathfrak{b})$ is an m-standardization adapted to $\mathcal{J}$ with respect to $Z$. The next result assures the existence of such a centers.

\begin{prop}\label{lema:existenciaestandadaptada}
Assume that $\text{Inv}_\mathcal{J}>0$. Then, there are codimension two combinatorial centers of blowing-up adapted to $\mathcal{J}$.
\end{prop}
\begin{proof}
By definition $\text{Inv}_\mathcal{J}>0$ if and only if $\Omega_{\mathcal{J}}\ne \emptyset$. Fix an element $Z \in \Omega_{\mathcal{J}}$ and let us see that there are m-standardizations adapted to $\mathcal{J}$ with respect to $Z$. In view of Lemma \ref{lema:estandadaptada}, it is enough to prove the existence of an m-standardization $(\mathcal{O},\mathfrak{b})$ adapted to $\mathcal{J}$ at a corner point $p\in Z$. Fix any corner point $p\in Z$. Since $Z$ is uncoupled for $\mathcal{J}$, we can assume, up to exchanging the indices $i$ and $j$, that 
$$
\ell_i=\lambda_{p,i}-\mu_{p,i} >0, \; \text{ and }\; \ell_j=\mu_{p,j}-\lambda_{p,j} >0.
$$
Take $\alpha_p:I_p \to \mathbb{R}_+$ to be a map such that $\alpha_{p,i}=\ell_i$ and $\alpha_{p,j}=\ell_j$, and take the affine coordinate system $\mathbf{u}_p=\mathbf{x}_p^{\alpha_p}$ defined in $V^\star_p$. Any m-standardization extending $\mathbf{u}_p$ is adapted to $\mathcal{J}$ with respect to $Z$ at the point $p$ because of the definition of $\alpha_p$. Moreover, such an extension exists as a consequence of Proposition \ref{prop:extension}.
\end{proof}

We conclude by applying finitely many times the following result.

\begin{prop}
Let $\mathcal{J}=(\mathfrak{m},\mathfrak{n})$ be an m-ideal with $\text{Inv}_{\mathcal{J}}>0$. Given an m-combinatorial center of blowing-up $\xi=(Z,\mathcal{O},\mathfrak{b})$ adapted to $\mathcal{J}$, the blowing-up $\pi_\xi:\mathcal{M}_\xi \to \mathcal{M}$ centered at $\xi$ satisfies that $\text{Inv}_{\pi_\xi^*\mathcal{J}}=\text{Inv}_{\mathcal{J}}-1$.
\end{prop}

\begin{proof}
Let us write $\mathcal{Z}_\xi=\mathcal{Z}_{\mathcal{M}_\xi}$, $\pi=\pi_\xi$, $E_\infty=\pi^{-1}(Z)$ and $I_Z=\{i,j\}$. Denote also 
$$
\mathcal{L}_\mathfrak{m}=\{\lambda_p\}_{p\in \mathcal{Z}^0}, \; \mathcal{L}_\mathfrak{n}=\{\mu_p\}_{p\in \mathcal{Z}^0}, \quad \mathcal{L}_{\pi^*\mathfrak{m}}=\{\lambda'_{p'}\}_{p'\in \mathcal{Z}_\xi^0}, \; \mathcal{L}_{\pi^*\mathfrak{n}}=\{\mu'_{p'}\}_{p'\in \mathcal{Z}_\xi^0}.
$$
Let us fix an element $T \in \Omega_{\mathcal{J}}$ different from $Z$. Denote by $S_T$ to the stratum in $\mathcal{S}_{\mathcal{M}}$ such that $\overline{S}_T=T$. The adherence $T'$ of $\pi^{-1}(S_T)$ is a codimension two geometrical center having index set $I_{T'}=I_T=\{r,s\}$. For any corner point $p'\in T'$, if we write $p=\pi_\xi(p')$, we have that
$$
\lambda_{p',r}=\lambda_{p,r}, \quad \mu_{p',r}=\mu_{p,r}, \quad \lambda_{p',s}=\lambda_{p,s}, \quad \mu_{p',s}=\mu_{p,s},
$$
in view of the relation between $\lambda_{p'}, \lambda_p$ and $\mu_{p'}, \mu_p$ established in Equation \eqref{eq:mfuncionexplosion}, and the expression of $B_{p'}$ given in Equations \eqref{eq:identidad} and \eqref{eq:blowup}. Then $T'$ is also a codimension two combinatorial center in $\mathcal{M}_\xi$ uncoupled for $\pi^*\mathcal{J}$, that is $T'\in \Omega_{\pi^*\mathcal{J}}$. We want to see that there are not more elements in $\Omega_{\pi^*\mathcal{J}}$ than these ones, that is, there is no codimension two combinatorial geometrical center $Z'\subset E_\infty$ uncoupled for $\pi^*\mathcal{J}$. 

Take a codimension two combinatorial geometrical center $Z'\subset E_\infty$ and a point $p'\in \mathcal{Z}_\xi^0(Z')$. In view of Lemma \ref{lema:desacoplado}, it is enough to prove that $Z'$ is not uncoupled for $\pi^*\mathcal{J}$ at $p'$. More precisely, if we write $I_{Z'}=\{k,\infty\}$, we want to show that 
$$
(\lambda_{p',k}-\mu_{p',k})(\lambda_{p',\infty}- \mu_{p',\infty}) \geq 0.
$$
Let us consider $p=\pi(p')$ and the local data $\alpha_p\in \Lambda_{(\mathcal{O},\mathfrak{b})}$. The corner point $p$ belongs to $Z$, and the affine coordinates $\mathbf{u}_p\in \mathfrak{b}$ define a local m-standardization adapted to $\mathcal{J}$ with respect to $Z$ at $p$, that is, we have the relation $\alpha_{p,j}(\lambda_{p,i}-\mu_{p,i})+\alpha_{p,i}(\lambda_{p,j}-\mu_{p,j})=0$. We know that $I_{p'}=I_p\setminus \{j\}\cup \{\infty\}$, up to exchanging the indices $i$ and $j$. Hence, the expression of $B'=B_{p'}:I_{p}\times I_{p'} \to \mathbb{R}_+$ is, using Equation \eqref{eq:blowup}, as follows:
$$
 B'_{\ell\ell}=1, \text{ for } \ell \in I_p\setminus \{j\}, \quad B'_{j\infty}=1, \quad B'_{i\infty}=\frac{\alpha_{p,j}}{\alpha_{p,i}}=\frac{\mu_{p,j}-\lambda_{p,j}}{\lambda_{p,i}-\mu_{p,i}},
$$ 
and $B'_{rs}=0$, otherwise. By Equation \eqref{eq:mfuncionexplosion} we get the relations $\lambda_{p',k}=\lambda_{p,k}$, $\mu_{p',k}=\mu_{p,k}$, and
$$
\lambda_{p',\infty}=\lambda_{p,j}+B'_{i,\infty}\lambda_{p,i}=\frac{\lambda_{p,i}\mu_{p,j}-\lambda_{p,j}\mu_{p,i}}{\lambda_{p,i}-\mu_{p,i}}=\mu_{p,j}+B'_{i,\infty}\mu_{p,i}=\mu_{p',\infty}.
$$
Thus $(\lambda_{p',k}-\mu_{p',k})(\lambda_{p',\infty}- \mu_{p',\infty}) =0$, and we are done.
\end{proof}
%

\subsection{Stratified Reduction of Singularities in Monomial Manifolds}
We use the result of principalization of m-ideals in order to prove the following statement:

\begin{prop} \label{prop:redsing}
Let us consider a generalized analytic manifold $\mathcal{M}=(M,\mathcal{G}_M)$ admitting a monomial atlas $\mathfrak{a}$. Given a generalized analytic function $f\in \mathcal{G}_M(M)$, there is an m-star $\sigma \in \mathcal{V}^m_{(\mathcal{M},\mathfrak{a})}$ such that $f'= f  \circ \sigma$ is of stratified monomial type.
\end{prop}

More precisely, we associate to $f$ a finitely generated m-ideal $\mathcal{J}_f$, and we prove that the principalization of $\mathcal{J}_f$ gives rise to the stratified reduction of singularities of $f$.

\strut

Let us consider a monomial generalized analytic manifold $(\mathcal{M},\mathfrak{a})$, and let $f \in \mathcal{G}_M(M)$ be a generalized analytic function in $\mathcal{M}$, where $\mathcal{M}=(M,\mathcal{G}_M)$. Fix a corner point $q\in \mathcal{Z}^0$ and take the affine coordinates system $\mathbf{x}_q \in \mathfrak{a}$. For simplicity, we will use from now on the notation $\Delta_{f}^{\mathfrak{a},q}:=\text{Supp}_{q}(f;\mathbf{x}_{q})\subset \mathbb{R}^{I_q}$. Given $\lambda_q\in \Delta_f^{\mathfrak{a},q}$, it makes sense to define the m-function $\mathfrak{m}_{\lambda_q}$ as the one having  the collection of maps $\mathcal{L}_{\mathfrak{m}_{\lambda_q}}=\{\lambda_qC^{pq}\}_{p\in \mathcal{Z}^0}$ as a combinatorial data, in view of Remark \ref{rem:list} and Equation \eqref{eq:relacionsoporte}.  The \emph{m-ideal $\mathcal{J}_f$ associated to $f$} is the ideal sheaf generated by the finite set of m-functions
$$
G_f=\bigcup_{q\in\mathcal{Z}^0} \big\{\mathfrak{m}_{\lambda_q}: \; \lambda_q\in \Delta_{f,\min}^{\mathfrak{a},q}\big\}.
$$
By definition, notice that for any corner point $q\in \mathcal{Z}^0$, we have
\begin{equation} \label{eq:minimal}
(\Gamma_{G_f,q})^{\min}=\Delta_{f,\min}^{\mathfrak{a},q}.
\end{equation}

\begin{lema} \label{lema:m-idealimplicareduction}
Given an m-star $\tau: (\mathcal{M}',\mathfrak{a}')\to (\mathcal{M},\mathfrak{a})$, we have that $\tau^*\mathcal{J}_f=\mathcal{J}_{f'}$, where $f'=f \circ \tau$.
\end{lema}
\begin{proof}
In view of Equation \eqref{eq:restriccionestrella}, it is enough to prove that for any corner point $p'\in \mathcal{Z}^0_{\mathcal{M}'}$ the equality $(\Gamma_{\tau^*G_f,p'})^{\min}=(\Gamma_{G_{f'},p'})^{\min}$ holds.
Denote for short $\Gamma_1=\Gamma_{\tau^*G_f,p'}$ and $\Gamma_2=\Gamma_{G_{f'},p'}$

Fix a point $p'\in \mathcal{Z}^0_{\mathcal{M}'}$ and consider $p=\tau(p')$. Write $\Delta:=\Delta_{f, \min}^{\mathfrak{a},p}$, and let $B^\tau_{p'}\in \mathfrak{B}_\tau$ be the matrix of exponents codifying $\tau$ at the corner point $p'$. Let $\Theta:=\{\lambda_pB^\tau_{p'}; \; \lambda_p\in \Delta\}$.  We prove that both $\Gamma^{\min}_1$ and $\Gamma^{\min}_2$ are equal to $\Theta^{\min}$. 

\strut

\emph{Step 1:  $\Gamma^{\min}_2=\Theta^{\min}.\quad $} Recall by Equation \eqref{eq:minimal} that $\Gamma^{\min}_2=\Delta'$, where $\Delta':=\Delta_{f'\min}^{\mathfrak{a}',p'}$. Let $\Delta=\{\lambda^1,\lambda^2,\ldots,\lambda^k\} \subset \mathbb{R}^{I_p}$, that is, the function $f$ around the corner point $p$ has the finite presentation $f|_{V^\star_p}=\mathbf{x}_{p}^{\lambda^1}U_1+\mathbf{x}_{p}^{\lambda^2}U_2+\cdots \mathbf{x}_{p}^{\lambda^k}U_k$, where $U_i(p)\ne 0$, for all $i=1,2,\ldots,k$. Taking into account that $\tau(V^\star_{p'})\subset V^\star_{p}$, the function $f'=f \circ \tau$ is written in the chart $\mathbf{x}_{p'}\in \mathfrak{a}'$ as
$$
f'|_{V^{\star}_{p'}}=\mathbf{x}_{p'}^{\tilde{\lambda}^1}(U_1\circ \tau|_{V^{\star}_{p'}})+\mathbf{x}_{p'}^{\tilde{\lambda}^2}(U_2\circ \tau|_{V^{\star}_{p'}})+\cdots \mathbf{x}_{p'}^{\tilde{\lambda}^k}(U_k\circ \tau|_{V^{\star}_{p'}}),
$$
where $\tilde{\lambda}^i=\lambda^iB^\tau_{p'}$. Since $B^\tau_{p'}$ is an invertible matrix, we can assure that $\tilde{\lambda}^r\ne \tilde{\lambda}^s$ for any pair of different indices $r,s \in \{1,2,\ldots,k\}$. This implies that $\Delta'= \{\tilde{\lambda}^1,\tilde{\lambda}^2,\ldots,\tilde{\lambda}^k\}^{\min}$ by definition of minimal support of $f'$ at $p'$. We are done, since $\Theta=\{\tilde{\lambda}^1,\tilde{\lambda}^2,\ldots,\tilde{\lambda}^k\}$.

\strut

\emph{Step 2: } $\Gamma_1^{\min}=\Theta^{\min}.\quad$ Write $\Gamma_{G_f,p}=\Delta \cup \tilde{\Delta}$ with $\Delta \cap \tilde{\Delta}=\emptyset$.
By Equation \eqref{eq:minimal} we know that for any $\mu \in \tilde{\Delta}$ there exists $\lambda \in \Delta$ such that $\lambda<_d \mu$. Note that $\Gamma_1=\Theta \cup \tilde{\Theta}$, where $\tilde{\Theta}:=\{\mu B^\tau_{p'}; \; \mu\in \tilde{\Delta}\}$. Therefore, we need only to prove the following claim:\emph{
	If $\lambda,\mu: I_p \to \mathbb{R}_+$ satisfy $\lambda \leq _d \mu$, then $\lambda B^\tau_{p'} \leq _d \mu B^\tau_{p'}$.}

For that, it is enough to consider the case where $\tau$ is a single m-combinatorial blowing-up $\tau=\pi_{\xi}$ with center $\xi =(Z,\mathcal{O},\mathfrak{b})$. Denote, as usual, $E_\infty=\pi_\xi^{-1}(Z)$.  If $p'\notin E_\infty$ or equivalently $p\notin Z$, we have that $I_{p'}=I_{p}$ and $B^\tau_{p'}=D_{\mathds{1}_{I_p}}$, and we are done. Assume that $p'\in E_\infty$, and let $j \in I_Z$ be the index such that $I_{p'}\setminus \{\infty\}=I_p\setminus \{j\}$.  Denote $\lambda'=\lambda B^\tau_{p'}$ and $\mu'=\mu B^\tau_{p'}$.  By Equation \eqref{eq:blowup}, we have that
$\lambda'_{\ell}=\lambda_{\ell}$, $\mu'_{\ell}=\mu_{\ell}$, and thus $\lambda'_{\ell} \leq \mu'_{\ell}$, for all $\ell \in I_{p'}\setminus \{\infty\}$; whereas
$$
\lambda'_{\infty}=\sum_{\ell \in I_Z}\frac{\alpha_{p,j}}{\alpha_{p,\ell}}\lambda_{\ell} \leq \sum_{\ell \in I_Z}\frac{\alpha_{p,j}}{\alpha_{p,\ell}}\mu_{\ell}= \mu'_{\infty},
$$
where $\alpha_p \in \Lambda_{(\mathcal{O},\mathfrak{b})}$, as we wanted.
\end{proof}

\begin{proof}[Proof of Propostion \ref{prop:redsing}:]
In view of Theorem \ref{teo:principalizacion}, we can take an m-star $\sigma: (\mathcal{M}',\mathfrak{a}')\to (\mathcal{M},\mathfrak{a})$ such that $\sigma^*\mathcal{J}_f$ is locally principal. By Lemma \ref{lema:m-idealimplicareduction} we know also that $\sigma^*\mathcal{J}_f=\mathcal{J}_{f\circ\sigma}$. Hence, since $G_{f\circ\sigma}$ is a set of generators of $\mathcal{J}_{f\circ\sigma}$, we have $(\Gamma_{G_{f\circ\sigma},p'})^{\min}$ is a singleton for all $p'\in \mathcal{Z}^0_{\mathcal{M}'}$. Finally, by Equation \eqref{eq:minimal} we obtain 
$$
m_{p'}(f)=\# \Delta_{f'\min}^{\mathfrak{a}',p'}=\# (\Gamma_{G_{f\circ\sigma},p'})^{\min}=1
$$
for all $p'\in \mathcal{Z}^0_{\mathcal{M}'}$. Since $\mathfrak{a}$ is a monomial atlas, we know that $M'=\bigcup_{p'\in \mathcal{Z}^0_{\mathcal{M}'}} V^\star_{p'}$. Thus, given a stratum $S \in \mathcal{S}_{\mathcal{M}}$, there is a corner point $p'\in \mathcal{Z}^0_{\mathcal{M}'}$ such that $p'\in \bar{S}$. In view of the horizontal stability property for the monomial complexity established in Lemma \ref{lema:estabilidadhorizontal}, we get $m_S(f) \leq m_{p'}(f)=1$. We conclude that $f'$ is of stratified monomial type.
\end{proof}

\subsection{Proof of the Main Statement}
We end by proving the stratified reduction of singularities for generalized analytic functions, as stated in Theorem \ref{th:main}. 

\strut

Recall that we have a generalized analytic manifold $\mathcal{M}=(M,\mathcal{G}_M)$, and a generalized analytic function $f \in \mathcal{G}_M(M)$ in $\mathcal{M}$. Given a point $p\in M$, we want to prove the existence of an open neighbourhood $V\subset M$ of $p$ and also the existence of a finite sequence of blowing-ups
$$
\sigma:	\mathcal{M}_r
	\stackrel{\pi_{r-1}}{\longrightarrow}
	\mathcal{M}_{r-1}
	\stackrel{\pi_{r-2}}{\longrightarrow} 
	\cdots
	\stackrel{\pi_0}{\longrightarrow}
	\mathcal{M}_0=(V,\mathcal{G}_M|_V),
$$
such that $f'=f\circ\sigma$ is of stratified monomial type in $\mathcal{M}_r$. Moreover, we are going to see that it can be done by taking blowing-ups with combinatorial centers of codimension two.

\strut

Fix a point $p\in M$, let $S$ be the stratum of $\mathcal{S}_\mathcal{M}$ containing $p$, and let us write $k:=\dim S$. We denote $e:=n-k$, where $n$ is the dimension of $M$. Take a local chart 
$$
\varphi:V \to \mathbb{R}^{k} \times \mathbb{R}_+^{e},
$$
centered at $p$. Note that  $\varphi(S)=\mathbb{R}^k \times \{\mathbf{0}\}$. Assume that the minimal support of $f$ along $S$, with respect to $\varphi$, is $\Delta_0:=\text{Supp}_{\min,S}(f;\varphi)=\{\lambda^1,\lambda^2,\ldots,\lambda^t\} \subset \mathbb{R}^e_+$. That is
$$
f\circ\varphi^{-1}=\mathbf{z}^{\lambda^1}U_1(\mathbf{y},\mathbf{z})+\mathbf{z}^{\lambda^2}U_2(\mathbf{y},\mathbf{z})+\cdots+\mathbf{z}^{\lambda^t}U_t(\mathbf{y},\mathbf{z}),
$$
where $\mathbf{y}=(y_1,y_2,\ldots,y_k)$ and $\mathbf{z}=(z_1,z_2,\ldots,z_e)$ are the natural coordinates in $\mathbb{R}^k$ and $\mathbb{R}^e$, respectively, $U_i(\mathbf{y},\mathbf{0})\not\equiv 0$, for all $i=1,2,\ldots,t$, and the elements of $\Delta_0$ are two-by-two incomparable with respect to the division order $\leq_d$ in $\mathbb{R}^e$.  

\strut

Consider the m-corner $(\mathbb{G}_e,\mathbf{z})$, and the generalized analytic function $\bar{f}=\mathbf{z}^{\lambda^1}+\mathbf{z}^{\lambda^2}+\cdots+\mathbf{z}^{\lambda^t}$ defined in $\mathbb{R}^e_+$. By Proposition \ref{prop:redsing}, there is a sequence of m-combinatorial blowing-ups
$$
\bar{\sigma}: (\bar{\mathcal{M}}_r,\mathfrak{a}_r)\stackrel{\bar{\pi}_{r-1}}{\longrightarrow} (\bar{\mathcal{M}}_{r-1},\mathfrak{a}_{r-1}) \stackrel{\bar{\pi}_{r-2}}{\longrightarrow} \cdots \stackrel{\bar{\pi}_0}{\longrightarrow} (\bar{\mathcal{M}}_0,\mathfrak{a}_0)=(\mathbb{G}_e,\mathbf{z}),
$$
such that $\bar{f}'=\bar{f} \circ \bar{\sigma}$ is of stratified monomial type. Let us write $\tilde{\sigma}=\text{id} \times \bar{\sigma}$, and  $\mathcal{W}=(\mathbb{R}^k,{\bar{\mathcal{O}}_k})$, where the sheaf ${\bar{\mathcal{O}}_k}$ has been introduced in Example \ref{ejem:traslacion}. 
If we prove that the map
$$
\sigma=\varphi^{-1}\circ \tilde{\sigma}:  \mathcal{W} \times \bar{\mathcal{M}}_r\to (V,\mathcal{G}_V)
$$
is the composition of a finite sequence of combinatorial blowing-ups such that $f'=f\circ \sigma$ is of stratified monomial type, we are done.

\strut

\paragraph{Step 1:} We see first that $\sigma$ is a sequence of combinatorial blowing-ups. 

Given an index $i$ among $0$ and $r$, we write $\bar{\mathcal{M}}_i=(\bar{M}_{i},\mathcal{G}_{\bar{M}_{i}})$. We consider the product manifold $\mathcal{M}_i=\mathcal{W}\times \bar{\mathcal{M}}_i$, and we denote by $\mathcal{G}_{M_i}$ to the structural sheaf of $\mathcal{M}_i$. We want to prove that, for $i\ne r$, the map $\pi_i=\text{id} \times \bar{\pi}_i: \mathcal{M}_{i+1} \to \mathcal{M}_i$ is a combinatorial blowing-up. Let $\bar{\xi}_i=(\bar{Z}_i,\mathcal{O}_{\bar{M}_i})$ be the center of blowing-up of $\bar{\pi}_i$, and denote $\bar{\mathcal{N}}_i=(\bar{M}_i,\mathcal{O}_{\bar{M}_i})$. By definition of standardization, recall that $\mathcal{O}_{\bar{M}_i}^\epsilon=\bar{\mathcal{G}}_i$. Define $Z_i=\mathbb{R}^k \times \bar{Z}_i$, and let $\mathcal{O}_{M_i}$ be the structural sheaf of the standard analytic manifold $\mathcal{W}\times\bar{\mathcal{N}}_i$.  Note that $\mathcal{O}_{M_i}^\epsilon=\mathcal{G}_{M_i}$ and that $Z_i \in \mathcal{Z}_{\mathcal{M}_i}$, and thus $\xi_i=(Z_i,\mathcal{O}_{M_i})$ is a combinatorial center of blowing-up for $\mathcal{M}_i$. Finally, we have that $\pi_i$ is the blowing-up of $\mathcal{M}_i$ centered at $\xi_i$. Indeed, just note that the blowing-up centered at $Z_i$ of the standard analytic manifold $(M_{i},\mathcal{O}_{M_i})$ is $\pi_{Z_i}=\text{id}\times \pi_{\bar{Z}_i}$, where $\pi_{\bar{Z}_i}:\tilde{\bar{\mathcal{N}}}_i \to \bar{\mathcal{N}}_i$ is the blowing-up centered at $\bar{Z}_i$ of $\bar{\mathcal{N}}_i$.

To finish, we see that the composition $\varphi^{-1}\circ \pi_0: (M_1,\mathcal{G}_{M_1}) \to (V,\mathcal{G}_V)$ is a blowing-up of the generalized analytic manifold $(V,\mathcal{G}_V)$. The combinatorial geometrical center is $\tilde{Z}_0=\varphi^{-1}(Z_0)$ and the standardization is the sheaf  $\mathcal{O}_V$ given locally  at $q \in V$ by
$$
\mathcal{O}_{V,q}=\{ g\circ \varphi: \; g \in \mathcal{O}_{M_0,\varphi(q)}\}.
$$
Indeed, note that $M_0=\mathbb{R}^k \times \mathbb{R}^e_+$, and that $\mathcal{O}_{M_0}^\epsilon=\mathcal{G}_{M_0}=\mathcal{G}_{k,e}$, where $\mathcal{G}_{k,e}$ has been introduced in Example \ref{ejem:estructuramixta}. Since $\varphi$ is an isomorphism we have that $\mathcal{O}_V \subset \mathcal{G}_V$ and also that $\mathcal{O}_V^\epsilon=\mathcal{G}_V$. Then $\xi=(\tilde{Z}_0,\mathcal{O}_V)$ is a combinatorial center of blowing-up. Since the blowing-up with center at $\tilde{Z}_0$ of the standard analytic manifold $(V,\mathcal{O}_V)$ is $\pi_{\tilde{Z}_0}=\pi_{Z_0}\circ\varphi^{-1}$, we get $\pi_\xi=\pi_0\circ\varphi^{-1}$ as we wanted.

\strut

\paragraph{Step 2:} Let us see that the function $f': M_r \to \mathbb{R}$ is of stratified monomial type. 

Denote $\mathcal{Z}^0_r= \mathcal{Z}^0_{\bar{\mathcal{M}}_r}$. Recall that $\mathcal{M}_r=\mathcal{W}\times \bar{\mathcal{M}}_r$, then there is a bijection between $\mathcal{Z}^0_r$ and the strata of dimension $k$ in $\mathcal{S}_{\mathcal{M}_r}$ sending a corner point $q$ into the stratum $S_q=\mathbb{R}^k \times \{q\}$. Let us prove that 
\begin{equation} \label{eq:igualdad}
m_{S_q}(f')=m_q(\bar{f}')=1,
\end{equation}
for each $q \in \mathcal{Z}^0_r$. If we do it, we are done. Indeed, any other stratum in $S \in \mathcal{M}_r$ contains $S_q$ in its adherence for some $q\in \mathcal{Z}^0_{r}$, and by the horizontal stability for the monomial complexity stated in Lemma \ref{lema:estabilidadhorizontal} we have that $m_S(f') \leq m_{S_q}(f')=1$.

\strut

Fix a corner point $q\in\bar{M}_r$ and $\mathbf{z}_q\in \mathfrak{a}_r$. Let us prove Equation \eqref{eq:igualdad}.  Take $B_q \in \mathfrak{B}_{\bar{\sigma}}$ be the matrix of exponents representing $\bar{\sigma}$ at $q$ and denote $\bar{\sigma}\Delta_0=\{\tilde{\lambda}^1,\tilde{\lambda}^2,\ldots,\tilde{\lambda}^t\}$, where $\tilde{\lambda}^i=\lambda^iB_{p'}$ for all $i=1,2,\ldots, t$. The expression of $f'$ around $S_q$ is:
$$
f'|_{\mathbb{R}^k \times \bar{V}^\star_q}=\mathbf{z}_{q}^{\tilde{\lambda}^1}(U_1\circ \tilde{\sigma}|_{\mathbb{R}^k\times V^{\star}_{q}})+\mathbf{z}_{q}^{\tilde{\lambda}^2}(U_2\circ \tilde{\sigma}|_{\mathbb{R}^k\times V^{\star}_{q}})+\cdots \mathbf{z}_{q}^{\tilde{\lambda}^k}(U_k\circ \tilde{\sigma}|_{\mathbb{R}^k\times V^{\star}_{q}}),
$$
and we also have that $\bar{f}'|_{\bar{V}^\star_q}=\mathbf{z}_{q}^{\tilde{\lambda}^1}+\mathbf{z}_{q}^{\tilde{\lambda}^2}+\cdots \mathbf{z}_{q}^{\tilde{\lambda}^k}$. Applying the same arguments that we have used in the proof of the first step of Lemma \ref{lema:m-idealimplicareduction}, we get that 
$$
\text{Supp}_{\min,q}(\bar{f}';\mathbf{z}_q)=\text{Supp}_{\min,S_q}(f'; (\mathbf{y}, \mathbf{z}_q))=(\bar{\sigma}\Delta_0)^{\min}.
$$
Since $(\bar{\sigma}\Delta_0)^{\min}$ is a singleton, we conclude $m_{S_q}(f') =\ m_{q}(\bar{f}')=1$, as wanted.

\end{document}